\documentclass[a4paper,11pt]{article}
\usepackage[utf8]{inputenc}
\usepackage{fullpage}
\usepackage{verbatim}
\usepackage{subcaption}
\usepackage{amsmath,amssymb,amsthm}
\usepackage{nicefrac}
\usepackage{bm}
\usepackage[sort&compress,numbers]{natbib}
\usepackage{notoccite}
\usepackage{xcolor}
\usepackage{graphicx}
\usepackage{tikz}
\usetikzlibrary{arrows.meta,shapes.geometric,fit,backgrounds}
\definecolor{my-blue}{rgb}{0.0,0.0,0.6}
\definecolor{my-red}{rgb}{0.5,0.0,0.0}
\definecolor{my-green}{rgb}{0.0,0.5,0.0}
\usepackage[
  colorlinks=true,
  linkcolor=my-red,
  citecolor=my-green,
  urlcolor=my-blue
]{hyperref}
\hypersetup{
  pdftitle={Likelihood-based inference for birth-death processes with composite birth mechanisms},
  pdfauthor={Marko Lalovic, Nicos Georgiou, Istvan Z. Kiss}
}

\newcommand{\convd}{\xrightarrow{\mathcal{D}}}
\newcommand{\convas}{\xrightarrow{a.s.}}

\numberwithin{equation}{section}
\theoremstyle{plain}
\newtheorem{theorem}{\sc Theorem}[section]
\newtheorem{lemma}[theorem]{\sc Lemma}
\newtheorem{proposition}[theorem]{\sc Proposition}
\newtheorem{corollary}[theorem]{\sc Corollary}
\newtheorem{assumption}[theorem]{\bf Assumption}

\theoremstyle{remark}
\newtheorem{remark}[theorem]{Remark}
\newtheorem{example}[theorem]{\bf Example}

\newcommand{\be}{\begin{equation}}
\newcommand{\ee}{\end{equation}}

\def\bE{\mathbb{E}}

\def\bP{\mathbb{P}}

\def\E{\bE}
\def\P{\bP}

\begin{document}
\title{Likelihood-based inference for birth-death processes with composite birth mechanisms}
\author{
Marko Lalovic,\textsuperscript{1}
Nicos Georgiou,\textsuperscript{2}
and Istv\'an Z. Kiss\textsuperscript{1,3}
\medskip\\
{\small 1. Network Science Institute, Northeastern University London, London, United Kingdom}\\
{\small 2. Department of Mathematics, University of Sussex, Brighton, United Kingdom}\\
{\small 3. Department of Mathematics, Northeastern University, Boston, MA, USA}
}
\date{}
\maketitle

\begin{abstract}
We develop a likelihood-based inference for finite-state birth-death processes with composite birth rates, in which multiple distinct mechanisms contribute additively to the total birth intensity. Our main motivating example is an SIS epidemic model with pairwise and higher-order transmission. The process is observed through a single aggregate trajectory, and in the main setting of interest, birth events are unmarked. This creates a deconvolution problem in event space: the state is one-dimensional, but the mechanism underlying each birth is latent.

We formulate the inference under a Doob $h$-transformed $Q$-process, which is time-homogeneous and ergodic and which provides a time-homogeneous asymptotic surrogate for the law of the original process conditioned on long survival. We derive the corresponding conditional likelihood and study both the conditional maximum likelihood estimator and a quasi-maximum likelihood estimator which is based on a simplified working score. 

Under the Doob-transform law, we prove consistency and asymptotic normality for both estimators, with asymptotic covariance determined by the inverse Fisher and inverse Godambe information matrices, respectively. We also showcase a practical one-dimensional test for the presence of a specific higher-order birth mechanism.
\end{abstract}



\section{Introduction}\label{sec:intro}
Understanding how contagion spreads through populations remains a central challenge in mathematical epidemiology. Classical models typically assume that transmission occurs through pairwise contacts between individuals. However, many real-world spreading phenomena, from infectious diseases to social behaviours, propagate through group interactions in which multiple 
individuals act collectively~\cite{centola2007complex, hodas2014simple}. This has motivated the study of higher-order interaction structures, such as hypergraphs and simplicial complexes, where transmission can occur not only through pairwise contacts but also through simultaneous group interactions involving three or more individuals~\cite{iacopini2019simplicial, battiston2020networks, torres2021and, Kiss2023Insights, bick2023higher}.

These higher-order interactions have been shown to qualitatively reshape the dynamics of spreading processes~\cite{iacopini2019simplicial, ferraz2024contagion, ferraz2023multistability}, as well as synchronisation 
phenomena~\cite{tanaka2011multistable, millan2020explosive, skardal2020higher} and evolutionary game theory~\cite{alvarez2021, civilini2021dilemmas, civilini2024explosive, guo2025evolutionary}. Indeed, higher-order mechanisms on hypergraphs allowing for multi-body interactions give rise to a range of phenomena absent from pairwise models~\cite{battiston2021physics, bianconi2021higher, rosas2022disentangling}, including explosive transitions~\cite{kuehn2021universal}, a vanishing critical mass~\cite{iacopini2022group}, multi-stability~\cite{ferraz2023multistability, 
skardal2023multistability, zhang2024deeper}, and chaos~\cite{sun2023dynamic}. Crucially, the way group interactions are distributed across the population determines much of this behaviour~\cite{st2021universal, majhi2022dynamics, 
mancastroppa2023hyper, nandy2024degree, kim2024higher}: from a macroscopic viewpoint, the presence of hubs in higher-order structures significantly influences the onset and evolution of spreading processes~\cite{landry2020effect, st2022influential}.

Despite this rapid progress in model development, important theoretical and empirical gaps remain. Rigorous understanding of how mean-field models relate to their exact stochastic counterparts on hypergraphs is still lacking, as is 
methodology for validating these models against controlled or empirical data and for assessing whether higher-order interactions are identifiable from observations. A recent review of contagion dynamics on higher-order networks~\cite{ferraz2024contagion} emphasises that, despite clear opportunities for further theoretical advances, empirical validation through data or experiments remains scarce and the field lacks a well-defined roadmap.

This article addresses one such fundamental open question: can distinct transmission mechanisms be identified from a single observed trajectory of the aggregate process? Specifically, if we observe only the aggregate number of infections over time, without knowing which mechanism triggered each infection event, can we disentangle the contributions of pairwise spread from those of higher-order group transmission? While we frame this question in the context of higher-order contagion, the underlying inferential problem is more general and arises whenever transition rates decompose into multiple mechanisms, as in chemical reaction networks~\cite{gillespie1977exact, erban2020stochastic, mcquarrie1967stochastic}, population processes~\cite{kendall1948birth}, and models of social or information diffusion~\cite{masuda2021gillespie}.

In this article, we develop a rigorous likelihood-based framework for addressing this question. Our motivating example is the higher-order SIS epidemic on a symmetric hypergraph, whose infected-count process can be mapped to a birth-death chain. Accordingly, we consider birth-death processes on a finite state space $\{0, 1, \ldots, N\}$, where the birth rate from state $k$ takes the composite form $\lambda_k(\bm\beta) = \sum_{i=1}^{K} \beta_i f_i(k)$, with known structural functions $f_i$ and unknown intensity parameters $\beta_i > 0$. Each term $\beta_i f_i(k)$ corresponds to a distinct transmission mechanism, and the inferential challenge is to estimate the individual $\beta_i$ from observations of the aggregate state process $(X_t)_{t \geq 0}$ alone. We call this problem \emph{deconvolution in the event space}: the state trajectory is one-dimensional, but its increments arise from a mixture of latent event types. Working with a fully connected population ensures that only the total number of infected individuals matters, making the birth-death framework a natural choice.

We also assume continuous observation throughout. The birth-death process has a long history as a modelling framework, with well-established inferential theory for classical linear cases~\cite{Keiding1975, Crawford2014}, and parameter estimation under survival conditioning has been studied for branching processes~\cite{braunsteins2022parameter, braunsteins2025, hautphenne2025consistentestimationsubcriticalbirthanddeath} and epidemic models~\cite{britton2019stochastic}, though typically under the assumption of a single birth mechanism. The composite rate structure we consider arises naturally in models of contagion on higher-order networks~\cite{iacopini2019simplicial, Kiss2023Insights}, where existing work has relied on mean-field approximations or simulation-based fitting rather than formal likelihood inference.

A fundamental complication arises from the presence of an absorbing extinction state at $k = 0$. In many applications, particularly epidemiology, we observe only trajectories that have not gone extinct: outbreaks that persisted 
long enough to be detected and monitored. Naively applying standard likelihood methods to such data introduces selection bias, since observed trajectories are not representative of the full process. Moreover, the size of the bias depends on the epidemic parameters, which are not a priori available. This phenomenon is well understood in population biology, where the vast majority of species that have ever existed are now extinct~\cite{raup1993extinction}, yet we overwhelmingly study the 
survivors. We handle this complication through the Doob $h$-transform, which constructs an ergodic process from the transient dynamics of the original absorbing chain and provides an appropriate inferential measure for asymptotic analysis~\cite{DarrochSeneta1967, ColletMartinez2013, ChampagnatVillemonais2016}.

A second complication is that, in practice, one typically observes only the aggregate state of the system, the total number of infected individuals, without knowing which transmission mechanism triggered each infection event. Disentangling pairwise from higher-order contributions via a single observed signal is the central statistical challenge of this paper. 

A natural downstream question is whether higher-order transmission is statistically detectable at all. For a fixed birth parameter $\beta_i$, one may test
\[
H_0:\beta_i=0
\qquad \text{versus} \qquad
H_1:\beta_i>0,
\]
thereby asking whether mechanism $i$ is present or not. In the SIS example, the case $i=2$ corresponds to testing for triadic transmission. We focus on this one-dimensional testing problem, which already captures the question of whether higher-order transmission can be detected from one-dimensional count data alone.

The contribution of the paper is twofold. First, we develop a conditional likelihood framework based on the $Q$-process and establish asymptotic normality for both the conditional MLE and the QMLE. Second, we provide a practical one-dimensional test for the presence of higher-order transmission. More broadly, the paper provides a rigorous likelihood-based approach to event-space deconvolution in finite-state birth-death models with composite birth rates.

The remainder of the paper is organised as follows. Section~\ref{sec:model} introduces the model, the SIS motivating example, and the main results. Section~\ref{sec:uncond-lik} derives the unconditional likelihood and naive maximum likelihood estimator which is biased under the survival selection. Section~\ref{sec:survival} develops survival conditioning via the Doob $h$-transform and the $Q$-process. Section~\ref{sec:conditional-lik} derives the conditional likelihood, score functions, and the conditional MLE and QMLE. Section~\ref{sec:asymptotic} presents the asymptotic results, including testing for the presence of a specific birth mechanism. We conclude with a discussion in Section~\ref{sec:discussion}.

\section{Model and main results}\label{sec:model}
We consider a class of continuous-time birth-death processes with multiple birth mechanisms but a one-dimensional observed state. The process $(X_t)_{t \geq 0}$ takes values in $\{0,1,\dots,N\}$, where $X_t$ denotes the population size at time $t$. Upward jumps may arise through several distinct mechanisms, while downward jumps correspond to removals or recoveries.

Our inferential setting is therefore a deconvolution problem in event space, not in state space: we observe the aggregate process $(X_t)$, but we do not observe which birth mechanism generated a given upward jump. The population itself is homogeneous, what is hidden is the type of event producing the jump.

We first define the model and observation scheme, then state the main inferential conclusions. In particular, naive inference based on surviving trajectories is biased due to survival selection, whereas likelihood-based estimators constructed under the $Q$-process law are consistent and asymptotically normal. This asymptotic theory also provides a practical one-dimensional test for the presence of a specific higher-order birth mechanism.

\subsection{Running example: SIS epidemic with higher-order transmission}
Our main motivating example is the simplicial susceptible-infected-susceptible (SIS) epidemic model introduced in \cite{iacopini2019simplicial}. We consider a population of $N$ individuals, each of whom is either susceptible or infected. The process $X_t$ tracks the number of infected individuals at time $t$.

In addition to standard pairwise transmission, we allow higher-order infection mechanisms. Fix $K \ge 2$. For each $i \in \{1,\dots,K\}$, every unordered $i$-tuple of infected individuals exerts infection pressure on each susceptible at rate $\beta_i > 0$. On the complete hypergraph containing all simplices up to order $K$, when $X_t=k$ there are $\binom{k}{i}$ infected $i$-tuples and $N-k$ susceptible individuals. Hence the total birth rate takes the form
\begin{equation}\label{eq:sis-birth-rate}
\lambda_k(\bm\beta) =
\sum_{i=1}^K \beta_i \binom{k}{i}(N-k),
\end{equation}
so that
\[
f_i(k)=\binom{k}{i}(N-k), \qquad i=1,\dots,K.
\]

Recoveries occur independently at per-capita rate $\mu>0$, so the total death rate is
\begin{equation}\label{eq:sis-death-rate}
\mu_k(\mu)=\mu k.
\end{equation}

Thus the simplicial SIS model is a composite birth-death process with multiple birth mechanisms, while the observed state remains one-dimensional through the infected count $X_t$. Figure~\ref{fig:SCM_model} illustrates the infection mechanisms.

\begin{figure}[tbp]
\centering
\begin{tikzpicture}[
    x=1cm, y=1cm,
    font=\small,
    infected/.style={
        circle,
        draw=black,
        fill=black,
        minimum size=7pt,
        inner sep=0pt
    },
    susceptible/.style={
        circle,
        draw=black,
        fill=white,
        line width=0.5pt,
        minimum size=7pt,
        inner sep=0pt
    },
    mechanismarrow/.style={
        -{Stealth[length=5pt,width=5pt]},
        line width=0.9pt
    },
    simplexedge/.style={
        draw=black,
        line width=0.6pt
    },
    simplexface/.style={
        fill=black!8,
        draw=none
    },
    localhalo/.style={
        fill=black!10,
        draw=none
    },
    figlabel/.style={
        font=\footnotesize,
        align=center
    },
    legendlabel/.style={
        font=\footnotesize,
        anchor=west
    }
]

\begin{scope}[shift={(0,0)}]
    \fill[localhalo] (0.5,0) ellipse (0.7 and 0.30);
    \node[infected] at (0,0) {};
    \node[susceptible] at (1,0) {};

    \draw[mechanismarrow] (1.45,0) -- (1.9,0);

    \fill[localhalo] (2.9,0) ellipse (0.7 and 0.30);
    \node[infected] at (2.4,0) {};
    \node[infected] at (3.4,0) {};

    \node[figlabel] at (1.7,-0.62) {Type-1 birth};
\end{scope}

\begin{scope}[shift={(5.4,0)}]
    \fill[localhalo] plot[smooth cycle, tension=0.7] coordinates {
        (-0.22,-0.18) (0.5,-0.25) (1.22,-0.18)
        (1.0,0.42) (0.5,1.02) (0.0,0.42)
    };
    \node[infected] at (0,0) {};
    \node[infected] at (1,0) {};
    \node[susceptible] at (0.5,0.8) {};

    \draw[mechanismarrow] (1.50,0.32) -- (1.95,0.32);

    \fill[localhalo] plot[smooth cycle, tension=0.7] coordinates {
        (2.18,-0.18) (2.9,-0.25) (3.62,-0.18)
        (3.4,0.42) (2.9,1.02) (2.4,0.42)
    };
    \node[infected] at (2.4,0) {};
    \node[infected] at (3.4,0) {};
    \node[infected] at (2.9,0.8) {};

    \node[figlabel] at (1.7,-0.62) {Type-2 birth};
\end{scope}

\begin{scope}[shift={(10.8,0)}]
    \fill[localhalo] plot[smooth cycle, tension=0.7] coordinates {
        (-0.28,-0.22) (0.55,-0.30) (1.38,-0.22)
        (1.15,0.50) (0.55,1.20) (-0.05,0.50)
    };
    \node[infected] at (0,0) {};
    \node[infected] at (1.10,0) {};
    \node[infected] at (0.55,0.92) {};
    \node[susceptible] at (0.55,0.32) {};

    \draw[mechanismarrow] (1.65,0.38) -- (2.10,0.38);

    \fill[localhalo] plot[smooth cycle, tension=0.7] coordinates {
        (2.32,-0.22) (3.15,-0.30) (3.98,-0.22)
        (3.75,0.50) (3.15,1.20) (2.55,0.50)
    };
    \node[infected] at (2.60,0) {};
    \node[infected] at (3.70,0) {};
    \node[infected] at (3.15,0.92) {};
    \node[infected] at (3.15,0.32) {};

    \node[figlabel] at (1.85,-0.62) {Type-3 birth};
\end{scope}

\begin{scope}[shift={(0.2,-2.9)}]
    \draw[simplexedge] (0,0) -- (1.4,0);
    \node[infected] at (0,0) {};
    \node[susceptible] at (1.4,0) {};
    \node[figlabel] at (0.7,-0.48) {$\beta_1$};
\end{scope}

\begin{scope}[shift={(4.0,-2.9)}]
    \coordinate (A) at (0,0);
    \coordinate (B) at (1.4,0);
    \coordinate (C) at (0.7,1.2);

    \fill[simplexface] (A) -- (B) -- (C) -- cycle;
    \draw[simplexedge] (A) -- (B) -- (C) -- cycle;

    \node[infected] at (A) {};
    \node[infected] at (B) {};
    \node[susceptible] at (C) {};

    \node[figlabel] at (0.7,-0.48) {$2\beta_1 + \beta_2$};
\end{scope}

\begin{scope}[shift={(7.9,-2.9)}]
    \coordinate (A) at (0,0);
    \coordinate (B) at (1.8,0);
    \coordinate (C) at (0.9,1.56);
    \coordinate (O) at (0.9,0.52);

    \fill[simplexface] (O) -- (A) -- (B) -- cycle;
    \fill[simplexface] (O) -- (B) -- (C) -- cycle;
    \fill[simplexface] (O) -- (C) -- (A) -- cycle;

    \draw[simplexedge] (A) -- (B) -- (C) -- cycle;
    \draw[simplexedge] (O) -- (A);
    \draw[simplexedge] (O) -- (B);
    \draw[simplexedge] (O) -- (C);

    \node[infected] at (A) {};
    \node[infected] at (B) {};
    \node[infected] at (C) {};
    \node[susceptible] at (O) {};

    \node[figlabel] at (0.9,-0.48) {$3\beta_1 + 3\beta_2 + \beta_3$};
\end{scope}

\begin{scope}[shift={(11.8,-3.1)}]
    \node[infected] at (0,0.30) {};
    \node[legendlabel] at (0.22,0.30) {infected};

    \node[susceptible] at (0,-0.15) {};
    \node[legendlabel] at (0.22,-0.15) {susceptible};
\end{scope}

\end{tikzpicture}
\caption{Higher-order infection mechanisms in the simplicial SIS model with $K=3$. A susceptible vertex (white) may become infected through pairwise interaction with one infected neighbour (rate $\beta_1$), through a triadic interaction involving two infected neighbours (rate $\beta_2$), or through a third-order interaction involving three infected neighbours (rate $\beta_3$). The total infection pressure is obtained by summing over all such local
configurations.}
\label{fig:SCM_model}
\end{figure}
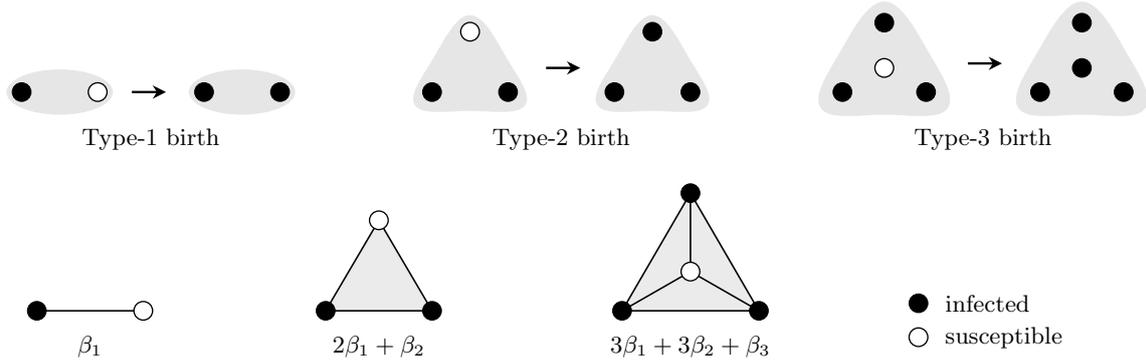

\subsection{Composite birth-death model}\label{subsec:bdp-def}
As noted above, due to the symmetry of the hypergraph, the number of infected in higher-order SIS epidemics is distributionally equal to a birth-death chain with multiple birth mechanisms and a homogeneous population. Therefore we will develop the inference based on this description.

We consider a continuous-time birth-death process $(X_t)_{t\ge0}$ on the finite state space
\[
\mathcal S=\{0,1,\dots,N\},
\]
where $X_t$ denotes the population size at time $t$. The state $0$ is absorbing and represents extinction.

For $k\in\{1,\dots,N-1\}$, the total birth rate is assumed to have the additive form
\be\label{eq:birth-rate}
\lambda_k(\bm\beta) :=
\sum_{i=1}^K \beta_i f_i(k) =
\bm\beta^\top \bm f(k),
\ee
where $K\ge2$ is the number of birth mechanisms,
\be
\bm\beta=(\beta_1,\dots,\beta_K)^\top,
\quad \text{ and } \quad
\bm f(k)=(f_1(k),\dots,f_K(k))^\top,
\ee
Here $\beta_1, \dots, \beta_K$ are unknown parameters, and the functions $f_i:\mathcal S\to[0,\infty)$ are assumed known structural functions determined by the model. We assume $f_i(0)=f_i(N)=0$ for all $i$, so that no births occur from the extinction state or when the system is at full capacity.

The total death rate from state $k$ to state $k-1$ is
\begin{equation}\label{eq:death-rate}
\mu_k(\mu) := \mu\,r(k),
\end{equation}
where $\mu>0$ is an unknown parameter and $r:\mathcal S\to[0,\infty)$ is another known structural function satisfying
$r(0)=0$ and $r(k)>0$ for all $k\ge1$.

The full parameter vector is
\[
\bm\theta=(\beta_1,\dots,\beta_K,\mu)^\top \in \Theta.
\]

For the main estimation problem we work on an open parameter set
\[
\Theta\subset(0,\infty)^{K+1},
\]
so that all mechanism intensities and the death parameter are strictly positive. We assume that the true parameter $\bm\theta_0=(\beta_{1,0},\dots,\beta_{K,0},\mu_0)^\top$ lies in the interior of $\Theta$. For the one-sided testing problem in Section~\ref{sec:testing}, we enlarge this to an open admissible set $\Theta_{\mathrm{test}}$ in which the coordinate under test may take either sign, provided the resulting birth intensities remain strictly positive.

\subsection{Unmarked births and event-space deconvolution}
We assume that the sample path $(X_t)_{0\le t\le T}$ is observed continuously on a finite time interval $[0,T]$. Each upward jump of $X_t$ corresponds to a birth event, and each downward jump corresponds to a death event. Thus the observed trajectory determines the jump times, waiting times, and jump directions. Figure~\ref{fig:BDP_sample_path} illustrates a sample path of such a process.

The key feature of the model is that births are \emph{unmarked}, when an upward jump occurs, we observe that a birth has taken place, but we do not observe which of the $K$ birth mechanisms generated it. By contrast, deaths are not decomposed into multiple types in the present formulation.

The process $(X_t)$ itself is one-dimensional and the population is homogeneous. The latent mechanism responsible for each birth event remains hidden. The central question is whether the parameters governing these distinct birth mechanisms can be recovered from continuous observation of the aggregate path alone.

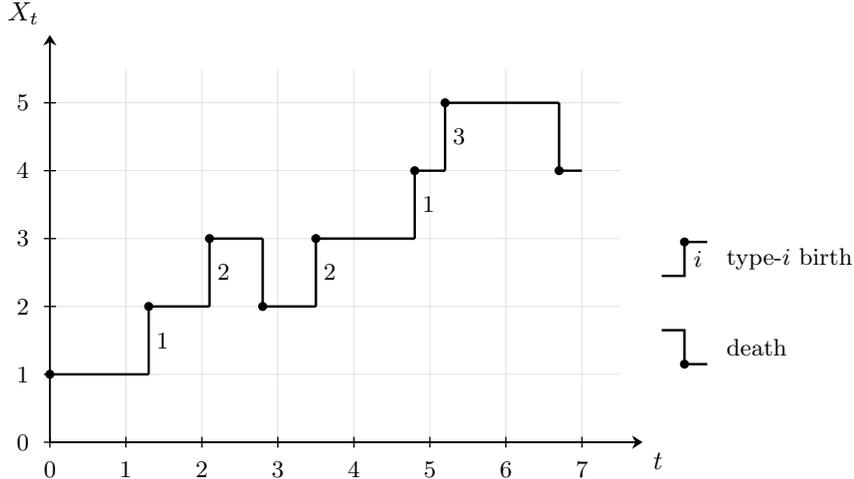
\begin{figure}[tbp]
\centering
\begin{tikzpicture}[
    x=1cm, y=0.9cm,
    font=\small,
    axis/.style={
        -{Stealth[length=4pt,width=5pt]},
        line width=0.7pt
    },
    gridline/.style={
        draw=black!12,
        line width=0.3pt
    },
    samplepath/.style={
        draw=black,
        line width=0.9pt
    },
    ticklabel/.style={
        font=\footnotesize
    },
    birthlabel/.style={
        font=\footnotesize,
        inner sep=1pt,
        fill=white,
        fill opacity=0.9,
        text opacity=1
    },
    legendtext/.style={
        font=\footnotesize,
        anchor=west
    }
]

\foreach \y in {1,2,3,4,5} {
    \draw[gridline] (0,\y) -- (7.5,\y);
}
\foreach \x in {1,2,3,4,5,6,7} {
    \draw[gridline] (\x,0) -- (\x,5.5);
}

\draw[axis] (0,0) -- (7.8,0) node[below right] {$t$};
\draw[axis] (0,0) -- (0,6.0) node[above left] {$X_t$};

\foreach \x in {0,1,2,3,4,5,6,7} {
    \draw[line width=0.5pt] (\x,0.08) -- (\x,-0.08);
    \node[ticklabel, below] at (\x,-0.14) {\x};
}

\foreach \y in {0,1,2,3,4,5} {
    \draw[line width=0.5pt] (-0.08,\y) -- (0.08,\y);
    \node[ticklabel, left] at (-0.14,\y) {\y};
}

\coordinate (p0)  at (0,1);
\coordinate (p1)  at (1.3,1);
\coordinate (p2)  at (1.3,2);
\coordinate (p3)  at (2.1,2);
\coordinate (p4)  at (2.1,3);
\coordinate (p5)  at (2.8,3);
\coordinate (p6)  at (2.8,2);
\coordinate (p7)  at (3.5,2);
\coordinate (p8)  at (3.5,3);
\coordinate (p9)  at (4.8,3);
\coordinate (p10) at (4.8,4);
\coordinate (p11) at (5.2,4);
\coordinate (p12) at (5.2,5);
\coordinate (p13) at (6.7,5);
\coordinate (p14) at (6.7,4);
\coordinate (p15) at (7.0,4);

\draw[samplepath]
    (p0) -- (p1)
    (p2) -- (p3)
    (p4) -- (p5)
    (p6) -- (p7)
    (p8) -- (p9)
    (p10) -- (p11)
    (p12) -- (p13)
    (p14) -- (p15);

\draw[samplepath]
    (p1) -- (p2)
    (p3) -- (p4)
    (p5) -- (p6)
    (p7) -- (p8)
    (p9) -- (p10)
    (p11) -- (p12)
    (p13) -- (p14);

\fill (p0)  circle (1.7pt);
\fill (p2)  circle (1.7pt);
\fill (p4)  circle (1.7pt);
\fill (p6)  circle (1.7pt);
\fill (p8)  circle (1.7pt);
\fill (p10) circle (1.7pt);
\fill (p12) circle (1.7pt);
\fill (p14) circle (1.7pt);

\node[birthlabel, right] at (1.36,1.50) {$1$};
\node[birthlabel, right] at (2.16,2.50) {$2$};
\node[birthlabel, right] at (3.56,2.50) {$2$};
\node[birthlabel, right] at (4.86,3.50) {$1$};
\node[birthlabel, right] at (5.26,4.50) {$3$};

\begin{scope}[shift={(8.35,2.45)}]
    \draw[samplepath] (-0.30,0.00) -- (0.00,0.00) -- (0.00,0.50) -- (0.30,0.50);
    \fill (0.00,0.50) circle (1.7pt);
    \node[birthlabel, right] at (0.08,0.25) {$i$};
    \node[legendtext] at (0.42,0.25) {type-$i$ birth};

    \draw[samplepath] (-0.30,-0.80) -- (0.00,-0.80) -- (0.00,-1.30) -- (0.30,-1.30);
    \fill (0.00,-1.30) circle (1.7pt);
    \node[legendtext] at (0.42,-1.05) {death};
\end{scope}

\end{tikzpicture}
\caption{Illustration of a sample path of the composite birth-death process on $[0,7]$. Upward jumps correspond to births and downward jumps to deaths. The birth-mechanism labels shown beside upward jumps are latent and are included only for illustration. In the unmarked births setting studied in this paper, these labels are not observed. Filled circles indicate right-continuity.}
\label{fig:BDP_sample_path}
\end{figure}

\subsection{Main results}\label{subsec:main-results}
A central difficulty in epidemic applications is that the state 0 is absorbing. In practice, one typically observes outbreaks that survive long enough to be detected and monitored, whereas trajectories that become extinct quickly, often leave no usable record. As a result, treating an observed long trajectory as if it were sampled from the unconditional law $\mathbb P_{\bm\theta}$ leads to survival-selection bias. We therefore aim to carry out inference for trajectories observed under the survival-conditioned law
\be\label{eq:cond-measure}
\mathbb P_{\bm\theta}(\,\cdot\,\mid \tau_0>T),
\qquad
\tau_0:=\inf\{t>0:X_t=0\},
\ee
where $T$ denotes the observation horizon. Since this law is not time-homogeneous, we replace it by the associated Doob $h$-transformed $Q$-process, which provides a time-homogeneous asymptotic surrogate for the survival-conditioned measure; see Section~\ref{sec:survival} and Appendix~\ref{app:QvsConditional}. We denote the measure under the $Q$-process by $\widetilde \P_{\bm \theta_0}$.

We consider three estimators. The naive unconditional maximum likelihood estimator, denoted by $\hat{\bm\theta}_T$, is obtained from the likelihood under the original process and ignores survival selection. The conditional maximum likelihood estimator $\widetilde{\bm\theta}_T$ is defined from the likelihood under the $Q$-process, while the quasi-maximum likelihood estimator $\bar{\bm\theta}_T$ is based on a simplified working score obtained by omitting derivatives of the tilting factors arising in the Doob $h$-transform.

Our first message is that naive inference is biased when based on surviving trajectories. Our main theoretical results then show that, under the $Q$-process law, both $\widetilde{\bm\theta}_T$ and $\bar{\bm\theta}_T$ are asymptotically well behaved. The estimator $\widetilde{\bm\theta}_T$ is asymptotically normal with covariance given by the inverse Fisher information, whereas $\bar{\bm\theta}_T$ is consistent and asymptotically normal with covariance given by the inverse Godambe information.

The estimator definitions and score equations are given in Section~\ref{sec:conditional-lik}, while the full asymptotic statements appear in Section~\ref{sec:asymptotic}. 

The main limit theorems underpinning statistical tests based on a single observed trajectory are as follows.
\begin{enumerate}
\item Consistency and asymptotic normality for the MLE
\[
\widetilde {\bm\theta}_T \stackrel{P}{\longrightarrow} \bm\theta_0,
\qquad
\sqrt{T}\bigl(\widetilde{\bm\theta}_T-\bm\theta_0\bigr)
\convd
\mathcal N\!\bigl(0,\mathcal I(\bm\theta_0)^{-1}\bigr)
\quad
\text{under}
\quad 
\widetilde{\mathbb P}_{\bm\theta_0},
\]

    \item Consistency and asymptotic normality for the QMLE 
\[
\bar{\bm\theta}_T \stackrel{P}{\longrightarrow} \bm\theta_0,
\qquad
\sqrt{T}\bigl(\bar{\bm\theta}_T-\bm\theta_0\bigr)
\convd
\mathcal N\!\bigl(0,\mathcal G(\bm\theta_0)^{-1}\bigr)
\quad
\text{under}
\quad
\widetilde{\mathbb P}_{\bm\theta_0},
\]
\end{enumerate}
For the statements above, the Fisher information matrix $\mathcal I$ is defined in Section~\ref{sec:info_matrices}, and the Godambe information matrix $\mathcal G$ is defined there as well. The symbol $\mathcal N(\bm \mu, \bm \Sigma)$ denotes a multivariate normal distribution with mean vector $\bm \mu$ and covariance matrix $\bm \Sigma$. For the precise statements of the results and the necessary assumptions to make the result rigorous, see Theorems \ref{thm:clt-mle} and \ref{thm:clt-qmle} whilst their proofs can be found in Appendix \ref{sec:proofs}.

Our main application is a one-dimensional test for the presence of a specific birth mechanism. For a fixed $i\in\{1,\dots,K\}$, we consider
\[
H_0:\beta_i=0
\qquad\text{versus}\qquad
H_1:\beta_i>0.
\]

This results in a one-sided Wald test for the presence of mechanism $i$ based on an unconstrained estimator. In particular, in the simplicial SIS model, the case $i=2$ corresponds to testing for the presence of triadic transmission. The associated one-dimensional test statistic and its calibration are developed at the end of Section~\ref{sec:asymptotic}.

\section{Counting processes and unconditional likelihood}\label{sec:uncond-lik}
We now derive the likelihood under the original, unconditional law of the composite birth-death process. This yields the naive maximum likelihood estimator, which treats the observed trajectory as if it were sampled from
$\mathbb P_{\bm\theta}$. We then explain why this is inappropriate in epidemic applications, where typically inference is done from a single trajectory only after it has survived for a substantial period of time.

\paragraph{Counting-process representation.}
Let $N_t^{(+)}$ and $N_t^{(-)}$ denote the cumulative numbers of births and deaths on $[0,t]$, respectively. Then
\begin{equation}\label{eq:state-counting}
X_t = X_0 + N_t^{(+)} - N_t^{(-)}, \qquad t\ge0.
\end{equation}
Under $\mathbb P_{\bm\theta}$, these counting processes admit the random time-change representation
\begin{equation}\label{eq:random-time}
N_t^{(+)} =
Y^{+}\!\left(\int_0^t \lambda_{X_s}(\bm\beta)\,ds\right),
\qquad
N_t^{(-)} =
Y^{-}\!\left(\int_0^t \mu_{X_s}(\mu)\,ds\right),
\end{equation}
where $Y^{+}$ and $Y^{-}$ are independent unit-rate Poisson processes \cite{ABGK1993}. Equivalently,
\[
X_t = X_0 +
Y^{+}\!\left(\int_0^t \lambda_{X_s}(\bm\beta)\,ds\right) -
Y^{-}\!\left(\int_0^t \mu_{X_s}(\mu)\,ds\right).
\]
This provides a constructive definition of the process and underpins exact
simulation via the Gillespie algorithm.

\paragraph{Unconditional likelihood.}
We assume that the trajectory $(X_t)_{0\le t\le T}$ is observed continuously on a finite interval $[0,T]$. Since births are unmarked, the observed path records the jump times and whether each jump is upward or downward, but not which birth mechanism generated a given upward jump.

The log-likelihood of the observed path under the unconditional law $\mathbb P_{\bm\theta}$ is
\begin{equation}\label{eq:uncond-loglik}
\ell_T(\bm\theta) =
\int_0^T \log \lambda_{X_{t-}}(\bm\beta)\, dN_t^{(+)} +
\int_0^T \log \mu_{X_{t-}}(\mu)\, dN_t^{(-)} -
\int_0^T \bigl[\lambda_{X_t}(\bm\beta)+\mu_{X_t}(\mu)\bigr]\,dt.
\end{equation}
where $X_{t-}$ denotes the left limit of the process at time $t$, that is, the state immediately before a possible jump at time $t$. The first term collects the contributions of upward jumps, the second the contributions of downward jumps, and the final integral is the compensator corresponding to the total jump intensity.

\begin{example}[Marked births]
The marked-birth setting is simpler, and it is used as a comparison case in this paper, since the theory works in that case as well. Consider a finite time horizon $T$. If the mechanism type of each birth event is observed, and $N^{(+)}_{i,t}$ denotes the number of type-$i$ births in $[0,t]$, then the birth term in \eqref{eq:uncond-loglik} decomposes as
\[
\sum_{i=1}^K \int_0^T \log\!\bigl(\beta_i f_i(X_{t-})\bigr)\,dN^{(+)}_{i,t}.
\]
For $i=1,\dots,K$ and $k=1,\dots,N-1$, define the statewise type-$i$ birth counts
\[
N_{i,k}^{(+)}(T):=\int_0^T \mathbf 1_{\{X_{t-}=k\}}\,dN_{i,t}^{(+)},
\]
the statewise death counts
\[
N_k^{(-)}(T):=\int_0^T \mathbf 1_{\{X_{t-}=k\}}\,dN_t^{(-)},
\qquad k=1,\dots,N,
\]
and the occupation times of state $k$, 
\[
T_k(T):=\int_0^T \mathbf 1_{\{X_t=k\}}\,dt,
\qquad k=0,\dots,N.
\]
Then, up to an additive constant independent of $\bm\theta$, the unconditional
log-likelihood becomes
\begin{align}
\ell_T^{\mathrm{mark}}(\bm\theta) &=
\sum_{i=1}^K\sum_{k=1}^{N-1}
\Big[
N_{i,k}^{(+)}(T)\log\!\big(\beta_i f_i(k)\big)
-\beta_i f_i(k)\,T_k(T)
\Big]
\label{eq:marked-uncond-loglik} 
\\
&\phantom{xxxxxxxxx}
+ \sum_{k=1}^{N}
\Big[
N_k^{(-)}(T)\log\!\big(\mu r(k)\big)
-\mu r(k)\,T_k(T)
\Big].
\nonumber
\end{align}
Hence the coordinates separate completely and, whenever the denominators are
positive, the naive unconditional MLE is available in closed form:
\[
\hat\beta_{i,T}^{\mathrm{mark}} =
\frac{\sum_{k=1}^{N-1} N_{i,k}^{(+)}(T)}
{\sum_{k=1}^{N-1} f_i(k)\,T_k(T)},
\quad i=1,\dots,K,
\text{ and } \quad
\hat\mu_T^{\mathrm{mark}} =
\frac{\sum_{k=1}^{N} N_k^{(-)}(T)}
{\sum_{k=1}^{N} r(k)\,T_k(T)}.
\]
Thus, once the birth marks are observed, each birth mechanism is estimated by its own event count divided by its integrated state-dependent exposure. The event-space deconvolution problem disappears at the level of the unconditional likelihood. \qed
\end{example}

\paragraph{Naive maximum likelihood estimator.}
The naive maximum likelihood estimator is defined by
\begin{equation}\label{eq:naive-mle}
\hat{\bm\theta}_T \in \arg\max_{\bm\theta\in\Theta} \ell_T(\bm\theta).
\end{equation}
When the maximiser lies in the interior of $\Theta$, it satisfies the score equation
\begin{equation}\label{eq:naive-score}
\nabla \ell_T(\bm\theta)=0.
\end{equation}

The estimator $\hat{\bm\theta}_T$ is the natural likelihood-based estimator under the original model. However, it is appropriate only if the observed path may genuinely be regarded as a typical realisation from the unconditional law
$\mathbb P_{\bm\theta}$.

\paragraph{Survival selection and bias.}
In epidemic applications, the observed trajectory is typically not a sample path from the unconditional law. Outbreaks that become extinct very quickly may never be detected, whereas trajectories that survive for long enough are much more likely to be observed and recorded. The available data are therefore subject to survival selection.

\begin{example}
Figure~\ref{fig:single_trajectory_survival} shows a simulated trajectory of the simplicial SIS model that survives over the full window $[0,T]$. In practice, trajectories of this kind are much more likely to be recorded than trajectories that go extinct shortly after initiation. The inferential sample is therefore selected through the event $\{\tau_0 > T\}$, and treating such data as if they were generated under the unconditional law $\mathbb P_{\bm\theta}$ leads to bias.
\end{example}

\begin{figure}[tbp]
\centering
\includegraphics[width=0.64\textwidth]{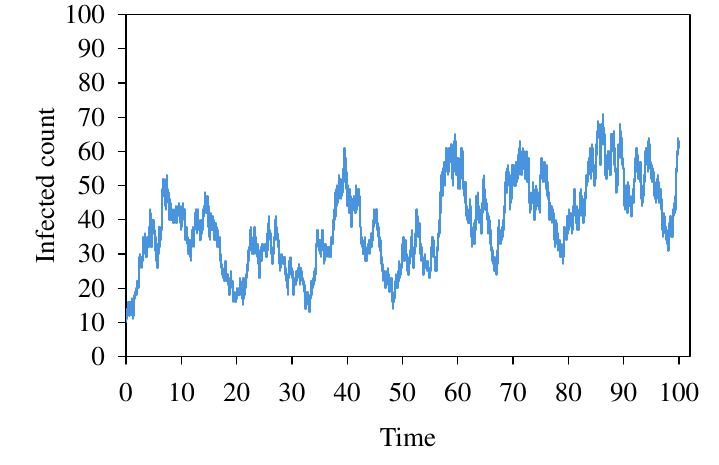}
\caption{Simulated surviving trajectory for the simplicial SIS model with $\bm\theta_0=(\beta_1,\beta_2,\mu) = (1.01/N,\,3.7/N^2,\,1)$, $N=100$, $X_0=10$, and $T=100$. In epidemic applications, trajectories of this type are far more likely to be observed than trajectories that go extinct quickly.}
\label{fig:single_trajectory_survival}
\end{figure}

If one applies the unconditional likelihood~\eqref{eq:uncond-loglik} to such data, the resulting estimator $\hat{\bm\theta}_T$ is generally biased, because it ignores the conditioning event
\[
\{\tau_0>T\},
\qquad
\tau_0:=\inf\{t>0:X_t=0\}.
\]
For moderate observation horizons and sufficiently strong transmission, the survival event may have probability close to one. However, in the present finite-state setting one has $\mathbb P_{\bm\theta}(\tau_0>T)\to0$ as $T\to\infty$, and the rate of this decay depends on the parameter values.

\begin{example}
Figure~\ref{fig:unconditional-mle-bias} illustrates the bias created by ignoring survival conditioning in the simplicial SIS model. In both panels, the estimators are computed from trajectories that survived up to time $T$, but the likelihood being maximised is the unconditional one. The resulting empirical distributions are therefore shifted away from the true parameter value.

The left panel shows this effect in the unmarked-birth setting, where birth mechanisms are latent and inference must disentangle them from the aggregate trajectory. The right panel shows that observing the birth marks does not remove the bias caused by survival selection. In that case the unconditional MLE is available in closed form through \eqref{eq:marked-uncond-loglik}, but it is still shifted away from the true parameter value. Thus marking removes the event-space deconvolution difficulty, but not the need to condition on survival.
\end{example}

\begin{figure}[tbp]
  \centering
  \begin{subfigure}{0.495\textwidth}
    \centering
    \caption{Unmarked births}
    \includegraphics[width=\textwidth]{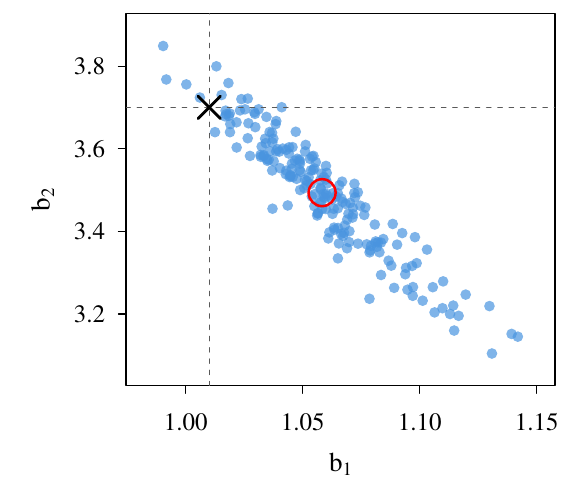}
  \end{subfigure}
  \hfill
  \begin{subfigure}{0.495\textwidth}
    \centering
    \caption{Marked births}
    \includegraphics[width=\textwidth]{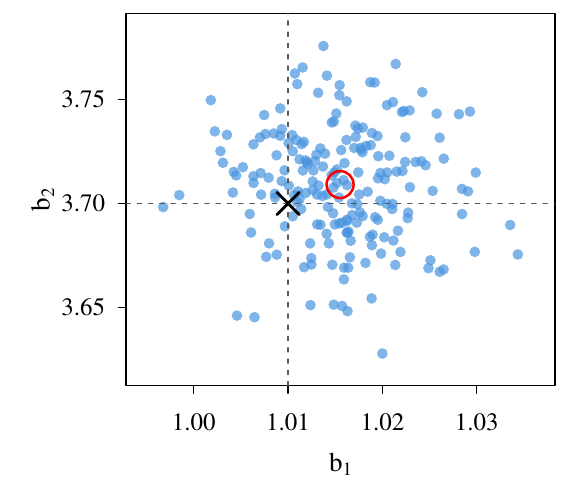}
  \end{subfigure}
\caption{Naive unconditional maximum likelihood estimates based on $200$ trajectories of the simplicial SIS model conditioned to survive up to time $T=1000$, with $N=100$ and $X_0=10$. The plotted coordinates are the scaled parameters $b_1=N\beta_1$ and $b_2=N^2\beta_2$; the true value is $(b_1,b_2)=(1.01,\,3.70)$. The black cross marks the true parameter and the red circle the empirical mean of the estimates. Panel (a) corresponds to unmarked births, and panel (b) to marked births. The two panels are shown on their own plotting scales. In both cases, ignoring survival conditioning leads to visible bias.}
\label{fig:unconditional-mle-bias}
\end{figure}

This motivates the conditional likelihood developed in the next section, where inference is formulated under the survival-conditioned law $\mathbb P_{\bm\theta}(\,\cdot\,\mid \tau_0>T)$.

\section{Survival conditioning and the \texorpdfstring{$Q$}{Q}-process}\label{sec:survival}
As discussed at the end of Section~\ref{sec:uncond-lik}, in epidemic applications the relevant inferential law is not the unconditional law $\mathbb P_{\bm\theta}$, but the law conditioned on survival up to the observation horizon,
\[
\mathbb P_{\bm\theta}(\,\cdot\,\mid \tau_0>T),
\qquad
\tau_0:=\inf\{t>0:X_t=0\}.
\]
This conditioning is conceptually natural, but technically inconvenient. Conditioning on a future event destroys time-homogeneity, so the resulting process is no longer described by a fixed generator. As a consequence, likelihood calculations and asymptotic arguments do not fit into the usual ergodic framework.

A standard way around this difficulty is to replace conditioning on survival by an equivalent change of dynamics via the Doob $h$-transform. This results in a new time-homogeneous birth-death process on the transient states $\{1,\dots,N\}$, called the $Q$-process, whose dynamics are tilted away from extinction. For large observation horizons, the $Q$-process provides the appropriate basis for conditional inference.

\subsection{Doob \texorpdfstring{$h$}{h}-transform and the \texorpdfstring{$Q$}{Q}-process}
Let $Q(\bm\theta)=(q_{k,j}(\bm\theta))_{k,j\in\mathcal S}$ denote the infinitesimal generator of the birth-death process $X$ on $\mathcal S=\{0,1,\dots,N\}$. Its entries are given by
\be
q_{k,j}(\bm\theta) =
\begin{cases}
\lambda_k(\bm\beta), & j=k+1,\\
\mu_k(\mu), & j=k-1,\\
-\bigl(\lambda_k(\bm\beta)+\mu_k(\mu)\bigr), & j=k,\\
0, & \text{otherwise},
\end{cases}
\qquad k\in\{1,\dots,N-1\},
\ee
with
\be
q_{0,j}(\bm\theta)=0,
\qquad j\in\mathcal S,
\ee
since state 0 is absorbing, and
\be
q_{N,j}(\bm\theta) =
\begin{cases}
\mu_N(\mu), & j=N-1,\\
-\mu_N(\mu), & j=N,\\
0, & \text{otherwise},
\end{cases}
\ee
since no births occur from state $N$.

Let $Q^+(\bm\theta)$ denote the restriction of $Q(\bm\theta)$ to the transient states $\{1,\dots,N\}$, obtained by deleting the row and column corresponding to the absorbing state $0$. Recall that $\Theta\subset(0,\infty)^{K+1}$ denotes the admissible parameter space introduced in Section~\ref{subsec:bdp-def}. We now impose conditions ensuring that, for every $\bm\theta\in\Theta$, the killed generator $Q^+(\bm\theta)$ is well behaved and the associated Doob $h$-transform is well defined.

\begin{assumption}[Admissible parameter region for the $Q$-process]\label{ass:irred_and_stab}
The parameter space $\Theta\subset(0,\infty)^{K+1}$ is open, and for every $\bm\theta\in\Theta$, the killed generator $Q^+(\bm\theta)$ is irreducible on $\{1,\dots,N\}$ and admits a simple maximal real eigenvalue $\gamma(\bm\theta)<0$ with strictly positive left and right eigenvectors.
\end{assumption}

\begin{remark}
\label{rem:admissible-region}
Assumption~\ref{ass:irred_and_stab} specifies the parameter region on which the killed chain and its Doob $h$-transform are well behaved. Since the state space is finite, irreducibility of $Q^+(\bm\theta)$ implies, by Perron-Frobenius theory, that the principal eigenvalue is simple and has strictly positive left and right eigenvectors. Thus, the assumption is mainly a restriction to those parameter values for which the transient chain on $\{1,\dots,N\}$ remains irreducible, so that the associated $Q$-process is well defined for every $\bm\theta\in\Theta$.
\end{remark}

Under Assumption~\ref{ass:irred_and_stab}, standard Perron-Frobenius theory for sub-stochastic generators~\cite{ColletMartinez2013} implies that, for every $\bm\theta\in\Theta$, there exists a unique maximal real eigenvalue $\gamma(\bm\theta)<0$, a strictly positive right eigenvector $\bm h_{\bm\theta}\in\mathbb R^N$, and a strictly positive left eigenvector $\bm v_{\bm\theta}\in\mathbb R^N$ such that

\be\label{eq:eigenproblem}
Q^+(\bm\theta)\bm h_{\bm\theta} =
\gamma(\bm\theta)\bm h_{\bm\theta},
\qquad
\bm v_{\bm\theta}^\top Q^+(\bm\theta) =
\gamma(\bm\theta)\bm v_{\bm\theta}^\top.
\ee

We normalise these eigenvectors so that
\be
\bm v_{\bm\theta}^\top \bm h_{\bm\theta}=1,
\qquad \bm\theta\in\Theta.
\ee

For notational convenience, we define the function $h_{\bm\theta}: \{0, 1, \ldots, N\} \to [0, \infty)$ by
\be
h_{\bm\theta}(k):= \begin{cases}
(\bm h_{\bm\theta})_k, & k\in\{1,\dots,N\},\\
0, & k=0.
\end{cases}
\ee

By extension $h_{\bm\theta}(0) = 0$, the transitions into the absorbing state are removed by the $h$-transform.

The Doob $h$-transform defines a new process $\widetilde X_t$, called the $Q$-process, with generator $\widetilde{\mathcal L}(\bm\theta)=(\widetilde q_{ij}(\bm\theta))_{1\le i,j\le N}$ given by
\be\label{eq:h-transform}
\widetilde q_{ij}(\bm\theta) =
\frac{h_{\bm\theta}(j)}{h_{\bm\theta}(i)}\,q_{ij}(\bm\theta),
\quad i\neq j,
\qquad
\widetilde q_{ii}(\bm\theta) =
-\sum_{j\neq i}\widetilde q_{ij}(\bm\theta).
\ee
Since $h_{\bm\theta}(0)=0$, the transformed rate into state 0 vanishes, and the $Q$-process evolves on $\{1,\dots,N\}$.

The $Q$-process is again a birth-death chain. Its birth and death intensities are obtained by tilting the original rates with ratios of $h_{\bm\theta}$. For $k=1,\dots,N-1$, define
\be
R^+_{\bm\theta}(k) :=
\frac{h_{\bm\theta}(k+1)}{h_{\bm\theta}(k)},
\qquad
\widetilde\lambda_k(\bm\theta) =
\lambda_k(\bm\beta)\,R^+_{\bm\theta}(k).
\ee
For $k=2,\dots,N$, define
\be
R^-_{\bm\theta}(k) :=
\frac{h_{\bm\theta}(k-1)}{h_{\bm\theta}(k)},
\qquad
\widetilde\mu_k(\bm\theta) =
\mu\,r(k)\,R^-_{\bm\theta}(k).
\ee
At the lower boundary,
\be
\widetilde\mu_1(\bm\theta) =
\mu\,r(1)\frac{h_{\bm\theta}(0)}{h_{\bm\theta}(1)} = 0,
\ee
so under the $Q$-process there is no death transition from state $1$.
These tilt factors bias the dynamics away from extinction by modifying the
transition rates through ratios of $h_{\bm\theta}$.

Under Assumption~\ref{ass:irred_and_stab}, the $Q$-process is positive recurrent on $\{1,\dots,N\}$ and admits the unique invariant distribution 
\be\label{eq:qsd}
\widetilde\pi_{\bm\theta}(k)=v_{\bm\theta}(k)h_{\bm\theta}(k).
\ee

The transformed process is again time-homogeneous and ergodic. This restores the setting needed for likelihood-based inference and for the martingale and ergodic arguments used later.

For killed Markov chains on a finite state space, the law of the process conditioned on survival up to a large horizon is well approximated, on any fixed time window, by the law of the Doob $h$-transformed process. In our setting, this means that the survival-conditioned measure
\[
\mathbb P_{\bm\theta}^{(T)}(\,\cdot\,) :=
\mathbb P_{\bm\theta}(\,\cdot\,\mid \tau_0>T)
\]
and the law $\widetilde{\mathbb P}_{\bm\theta}$ of the $Q$-process become indistinguishable on fixed observation windows as $T\to\infty$. 

For killed finite-state Markov chains, the $Q$-process captures the local dynamics of the original process conditioned on long survival. Proposition \ref{prop:asym-equivalence} shows that, on every fixed time window, the survival-conditioned law is asymptotically equivalent to the $Q$-process law.

Accordingly, in the remainder of the paper we develop and analyse inference under the $Q$-process itself, which provides a time-homogeneous surrogate for survival-conditioned dynamics.

A precise statement and proof of the required approximation are given in Appendix~\ref{app:QvsConditional}. Additional assumptions on the spectral quantities, needed for the asymptotic theory, are stated later.

\section{Conditional likelihood, scores, and estimators}\label{sec:conditional-lik}
We now specialise the unconditional likelihood-based inference developed in Section~\ref{sec:uncond-lik} to the survival-conditioned setting. Following Section~\ref{sec:survival}, we work under the Doob $h$-transformed $Q$-process, whose tilted birth and death rates are $\widetilde\lambda_k(\bm\theta)$ and $\widetilde\mu_k(\bm\theta)$. The conditional likelihood has the same form as the unconditional likelihood, but is evaluated under the transformed dynamics. From this we obtain the full score for the conditional MLE, a simplified working score for the QMLE, and the associated Fisher and Godambe information matrices.

\begin{remark}
In this section and below, ``conditional'' always refers to likelihood-based
inference under the $Q$-process law $\widetilde{\mathbb P}_{\bm\theta}$, which
serves as the time-homogeneous asymptotic surrogate of
$\mathbb P_{\bm\theta}(\,\cdot\,\mid \tau_0>T)$; see
Appendix~\ref{app:QvsConditional}.
\end{remark}

Under the $Q$-process, the counting processes $N_t^{(+)}$ and $N_t^{(-)}$ have predictable intensities
\[
\widetilde\lambda_{\widetilde X_{t-}}(\bm\theta) =
\lambda_{\widetilde X_{t-}}(\bm\beta)\,
R_{\bm\theta}^+(\widetilde X_{t-}),
\qquad
\widetilde\mu_{\widetilde X_{t-}}(\bm\theta) =
\mu\,r(\widetilde X_{t-})\,
R_{\bm\theta}^-(\widetilde X_{t-}).
\]
Hence the conditional log-likelihood on $[0,T]$ is
\begin{equation}\label{eq:cond-loglik}
\widetilde\ell_T(\bm\theta) =
\int_0^T \log \widetilde\lambda_{\widetilde X_{t-}}(\bm\theta)\,dN_t^{(+)} +
\int_0^T \log \widetilde\mu_{\widetilde X_{t-}}(\bm\theta)\,dN_t^{(-)} -
\int_0^T \bigl(
\widetilde\lambda_{\widetilde X_t}(\bm\theta) +
\widetilde\mu_{\widetilde X_t}(\bm\theta)
\bigr)\,dt.
\end{equation}

This has the same form as the unconditional log-likelihood, but with the original birth and death rates replaced by their tilted counterparts under the $Q$-process.

\begin{example}[Marked births under the $Q$-process]
The marked-birth setting remains simpler also under survival conditioning. If the mechanism type of each birth event is observed, and $N^{(+)}_{i,t}$ denotes the number of type-$i$ births in $[0,t]$, then under the $Q$-process the type-$i$ birth intensity at state $k$ is
\[
\widetilde\lambda_{i,k}(\bm\theta) =
\beta_i f_i(k)\,R_{\bm\theta}^+(k), \qquad i=1,\dots,K,
\]
so that
\[
\widetilde\lambda_k(\bm\theta) =
\sum_{i=1}^K \widetilde\lambda_{i,k}(\bm\theta) = 
R_{\bm\theta}^+(k)\sum_{i=1}^K \beta_i f_i(k).
\]
Accordingly, the birth term in \eqref{eq:cond-loglik} decomposes as
\[
\sum_{i=1}^K
\int_0^T
\log\!\bigl(\beta_i f_i(\widetilde X_{t-})R_{\bm\theta}^+(\widetilde X_{t-})\bigr)\,
dN^{(+)}_{i,t}.
\]
Thus, under the marked-birth conditional likelihood, the birth mechanisms remain separated at the level of observed jumps, and the only difference relative to the unconditional case is the presence of the common tilt factor $R_{\bm\theta}^+$. \qed
\end{example}

\subsection{Conditional score and MLE}
We work under the $Q$-process law $\widetilde{\mathbb P}_{\bm\theta}$. Define the compensated birth and death martingales by
\[
dM_t^{(+)} :=
dN_t^{(+)}-\widetilde\lambda_{\widetilde X_{t-}}(\bm\theta)\,dt,
\qquad
dM_t^{(-)} :=
dN_t^{(-)}-\widetilde\mu_{\widetilde X_{t-}}(\bm\theta)\,dt.
\]

Since births and deaths do not occur simultaneously,
\[
\langle M^{(+)},M^{(-)}\rangle_t\equiv0.
\]

Moreover, under the $Q$-process there is no death transition from state 1, since
\[
\widetilde\mu_1(\bm\theta) = \mu\,r(1)\frac{h_{\bm\theta}(0)}{h_{\bm\theta}(1)} = 0.
\]
Accordingly, all death-side log-intensity derivatives are taken only on the state set $\{2,\dots,N\}$.

Differentiating \eqref{eq:cond-loglik} with respect to $\bm\theta$ gives the score vector
\[
U_T(\bm\theta) =
\bigl(U_T^{(\bm\beta)}(\bm\theta),\,U_T^{(\mu)}(\bm\theta)\bigr),
\]
with components
\begin{align}
U_T^{(\bm\beta)}(\bm\theta) &=
\int_0^T g_{\bm\beta}^{(+)}(\widetilde X_{t-};\bm\theta)\,dM_t^{(+)} +
\int_0^T g_{\bm\beta}^{(-)}(\widetilde X_{t-};\bm\theta)\,dM_t^{(-)},
\label{eq:score-beta-full}\\
U_T^{(\mu)}(\bm\theta) &=
\int_0^T g_{\mu}^{(+)}(\widetilde X_{t-};\bm\theta)\,dM_t^{(+)} +
\int_0^T g_{\mu}^{(-)}(\widetilde X_{t-};\bm\theta)\,dM_t^{(-)},
\label{eq:score-mu-full}
\end{align}
where
\begin{align}
g_{\bm\beta}^{(+)}(k;\bm\theta) &:=
\nabla_{\bm\beta}\log \widetilde\lambda_k(\bm\theta) =
\frac{\bm f(k)}{\bm\beta^\top\bm f(k)} +
\nabla_{\bm\beta}\log R_{\bm\theta}^+(k), \label{eq:full-score-beta-plus}\\[0.25em]
g_{\bm\beta}^{(-)}(k;\bm\theta) &:=
\nabla_{\bm\beta}\log \widetilde\mu_k(\bm\theta) =
\nabla_{\bm\beta}\log R_{\bm\theta}^-(k), \label{eq:full-score-beta-minus}\\[0.25em]
g_{\mu}^{(+)}(k;\bm\theta) &:=
\partial_\mu\log \widetilde\lambda_k(\bm\theta) =
\partial_\mu\log R_{\bm\theta}^+(k), \label{eq:full-score-mu-plus}\\[0.25em]
g_{\mu}^{(-)}(k;\bm\theta) &:=
\partial_\mu\log \widetilde\mu_k(\bm\theta) =
\frac{1}{\mu} +
\partial_\mu\log R_{\bm\theta}^-(k). \label{eq:full-score-mu-minus}
\end{align}

and, by convention,
\[
g_{\bm\beta}^{(-)}(1;\bm\theta)=0,\qquad g_{\mu}^{(-)}(1;\bm\theta)=0.
\]

The key difference from the original model is that the tilt factors $R_{\bm\theta}^\pm$ depend on the full parameter vector. Consequently, the full score exhibits cross-sensitivity: even the birth score depends on $\mu$, and even the death score depends on $\bm\beta$. In particular, both birth and death martingales contribute to each score component.

Under $\widetilde{\mathbb P}_{\bm\theta_0}$, the score $U_T(\bm\theta_0)$ is a martingale. The conditional maximum likelihood estimator is defined as
\[
\widetilde{\bm\theta}_T
\quad\text{such that}\quad
U_T(\widetilde{\bm\theta}_T)=0.
\]

\begin{example}
Figures~\ref{fig:mle_consistency_unmarked} and~\ref{fig:mle_consistency_marked} illustrate the behaviour of the conditional MLE under the $Q$-process model as the observation horizon increases. In both settings, the estimator is computed from survival-conditioned trajectories, so the bias induced by naive unconditional inference is removed. As $T$ grows, the empirical clouds concentrate around the true parameter value, in agreement with the consistency proved later in the asymptotic theory.

\begin{figure}[tbp]
\centering
\begin{subfigure}{0.325\textwidth}
  \centering
  \caption{$T=200$}
  \includegraphics[width=\linewidth]{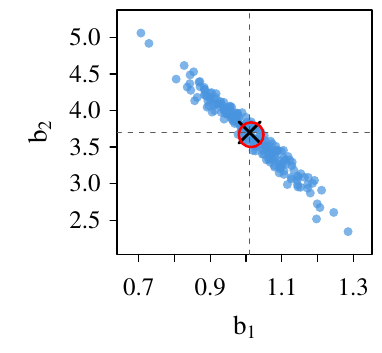}
\end{subfigure}
\hfill
\begin{subfigure}{0.325\textwidth}
  \centering
  \caption{$T=500$}
  \includegraphics[width=\linewidth]{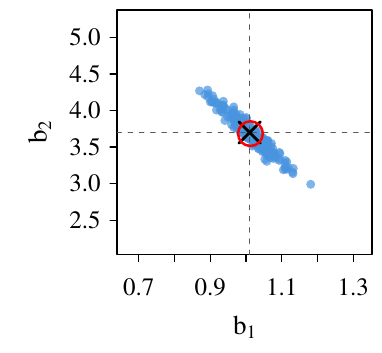}
\end{subfigure}
\hfill
\begin{subfigure}{0.325\textwidth}
  \centering
  \caption{$T=1000$}
  \includegraphics[width=\linewidth]{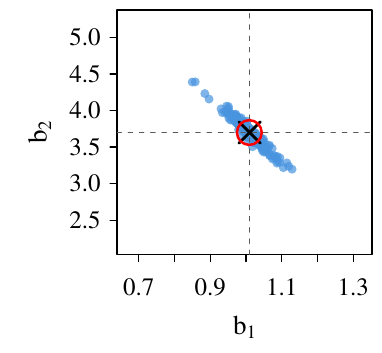}
\end{subfigure}
\caption{Conditional MLE for \emph{unmarked} births from $200$ survival-conditioned trajectories of the simplicial SIS model, with $N=100$ and $X_0=10$. The plotted coordinates are $b_1=N\beta_1$ and $b_2=N^2\beta_2$. The black cross marks the true value $(b_1,b_2)=(1.01,\,3.70)$, and the red circle the empirical mean of the estimates. As $T$ increases, the cloud contracts around the truth.}
\label{fig:mle_consistency_unmarked}
\end{figure}

\begin{figure}[tbp]
\centering
\begin{subfigure}{0.325\textwidth}
  \centering
  \caption{$T=200$}
  \includegraphics[width=\linewidth]{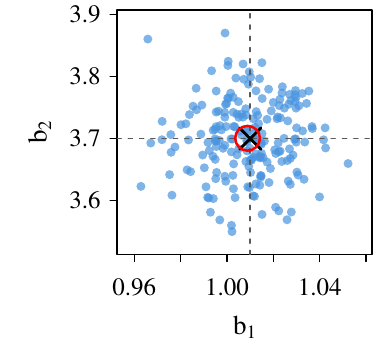}
\end{subfigure}
\hfill
\begin{subfigure}{0.325\textwidth}
  \centering
  \caption{$T=500$}
  \includegraphics[width=\linewidth]{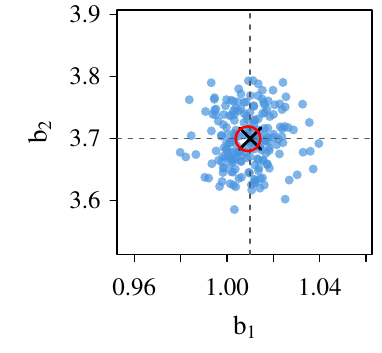}
\end{subfigure}
\hfill
\begin{subfigure}{0.325\textwidth}
  \centering
  \caption{$T=1000$}
  \includegraphics[width=\linewidth]{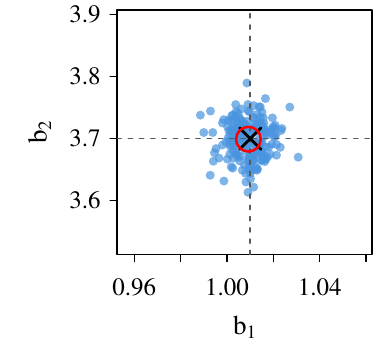}
\end{subfigure}
\caption{Conditional MLE for \emph{marked} births from $200$ survival-conditioned trajectories of the simplicial SIS model, with $N=100$ and $X_0=10$. The plotted coordinates are $b_1=N\beta_1$ and $b_2=N^2\beta_2$. The black cross marks the true value $(b_1,b_2)=(1.01,\,3.70)$, and the red circle the empirical mean of the estimates. The clouds are tighter than in Figure~\ref{fig:mle_consistency_unmarked} at each horizon, reflecting the additional information contained in the observed mechanism labels.}
\label{fig:mle_consistency_marked}
\end{figure}

In the unmarked setting, each upward jump must be attributed indirectly through the aggregate likelihood, so the estimates are more dispersed. The scatter also shows a pronounced ridge, indicating strong negative correlation between the estimators of $\beta_1$ and $\beta_2$. When birth types are observed, the additional information leads to visibly tighter clouds at every horizon. Thus marking does not change the survival-conditioned target of inference, but it makes the estimation problem substantially easier. The only remaining dependence between parameter components arises through the common tilt factors induced by the Doob $h$-transform. \qed
\end{example}

\subsection{Information matrices, working score and QMLE}\label{sec:info_matrices}
The asymptotic covariance structure is determined by the predictable quadratic variation of the full score. Write the parameter vector as
\[
\bm\theta=(\theta_1,\dots,\theta_{K+1})^\top = (\beta_1,\dots,\beta_K,\mu)^\top,
\]
and for each coordinate $\theta_a$ define the corresponding score component
\[
U_{a,T}(\bm\theta):=\partial_{\theta_a}\widetilde\ell_T(\bm\theta).
\]
From \eqref{eq:score-beta-full}--\eqref{eq:score-mu-full}, each score component has the form
\[
U_{a,T}(\bm\theta) =
\int_0^T \partial_{\theta_a}\log \widetilde\lambda_{\widetilde X_{t-}}(\bm\theta)\,dM_t^{(+)}
+
\int_0^T \partial_{\theta_a}\log \widetilde\mu_{\widetilde X_{t-}}(\bm\theta)\,dM_t^{(-)}.
\]
Since $\langle M^{(+)},M^{(-)}\rangle_t\equiv0$, the mixed birth--death bracket vanishes, and therefore
\begin{align}
\langle U_a,U_b\rangle_T &= 
\int_0^T
\partial_{\theta_a}\log \widetilde\lambda_{\widetilde X_{t-}}(\bm\theta)\,
\partial_{\theta_b}\log \widetilde\lambda_{\widetilde X_{t-}}(\bm\theta)\,
\widetilde\lambda_{\widetilde X_{t-}}(\bm\theta)\,dt
\label{eq:score-bracket-generic}
\\
&\qquad
+ \int_0^T
\partial_{\theta_a}\log \widetilde\mu_{\widetilde X_{t-}}(\bm\theta)\,
\partial_{\theta_b}\log \widetilde\mu_{\widetilde X_{t-}}(\bm\theta)\,
\widetilde\mu_{\widetilde X_{t-}}(\bm\theta)\,dt.
\nonumber
\end{align}
By ergodicity of the $Q$-process,
\[
\frac{1}{T}\langle U_a,U_b\rangle_T \convas \mathcal I_{ab}(\bm\theta),
\]
where the Fisher information matrix $\mathcal I(\bm\theta) = (\mathcal I_{ab}(\bm\theta))_{1\le a,b\le K+1}$ is given by
\begin{align}
\mathcal I_{ab}(\bm\theta) &=
\sum_{k=1}^{N-1}
\widetilde\pi_{\bm\theta}(k)\,\widetilde\lambda_k(\bm\theta)\,
\partial_{\theta_a}\log \widetilde\lambda_k(\bm\theta)\,
\partial_{\theta_b}\log \widetilde\lambda_k(\bm\theta)
\label{eq:Fisher-generic}
\\
&\qquad 
+ \sum_{k=2}^{N}
\widetilde\pi_{\bm\theta}(k)\,\widetilde\mu_k(\bm\theta)\,
\partial_{\theta_a}\log \widetilde\mu_k(\bm\theta)\,
\partial_{\theta_b}\log \widetilde\mu_k(\bm\theta).
\nonumber
\end{align}
In particular, for $a=\beta_i$ and $b=\mu$,
\begin{align}
\mathcal I_{\beta_i\mu}(\bm\theta)
&= \sum_{k=1}^{N-1}
\widetilde\pi_{\bm\theta}(k)\,\widetilde\lambda_k(\bm\theta)\,
g_{\beta_i}^{(+)}(k;\bm\theta)\,g_{\mu}^{(+)}(k;\bm\theta)
\label{eq:fisher-cross-coordinate}
\\
&\qquad 
+ \sum_{k=2}^{N}
\widetilde\pi_{\bm\theta}(k)\,\widetilde\mu_k(\bm\theta)\,
g_{\beta_i}^{(-)}(k;\bm\theta)\,g_{\mu}^{(-)}(k;\bm\theta).
\nonumber
\end{align}
Even though the mixed birth-death martingale bracket vanishes, the Fisher information matrix $\mathcal I(\bm\theta)$ need not be block diagonal because the transformed birth and death intensities themselves still depend on the full parameter vector through the tilt factors. 

The full score requires evaluation of the sensitivities $\nabla_{\bm\theta}\log R_{\bm\theta}^\pm(k)$, which in turn involves differentiating the Perron-Frobenius eigenvector of the killed generator. This is computationally expensive and may be numerically unstable when carried out repeatedly.

We therefore introduce the quasi-maximum likelihood estimator based on a simpler working score, obtained by omitting derivatives of the tilt factors in the score weights. Define
\begin{equation}\label{eq:working-score}
\bar g_{\bm\beta}(k;\bm\theta) :=
\frac{\bm f(k)}{\bm\beta^\top\bm f(k)},
\qquad
\bar g_\mu(k;\bm\theta) :=
\frac{1}{\mu},
\end{equation}

and
\begin{equation}\label{eq:working-score-vec}
\bar U_T^{(\bm\beta)}(\bm\theta) =
\int_0^T \bar g_{\bm\beta}(\widetilde X_{t-};\bm\theta)\,dM_t^{(+)},
\qquad
\bar U_T^{(\mu)}(\bm\theta) =
\int_0^T \bar g_{\mu}(\widetilde X_{t-};\bm\theta)\,dM_t^{(-)}.
\end{equation}
Writing
$
\bar U_T(\bm\theta) =
\bigl(\bar U_T^{(\bm\beta)}(\bm\theta),\,\bar U_T^{(\mu)}(\bm\theta)\bigr),
$
we define the quasi-maximum likelihood estimator by
\[
\bar{\bm\theta}_T
\quad\text{such that}\quad
\bar U_T(\bar{\bm\theta}_T)=0.
\]

Although the derivatives of the tilt factors are omitted from the score, the tilted intensities themselves remain in the compensators. The resulting estimator is therefore simpler computationally, while still targeting the same survival-conditioned model.

For the QMLE, the asymptotic covariance is determined by the Godambe information
\begin{equation}\label{eq:Godambe}
\mathcal G(\bm\theta) =
\mathcal H(\bm\theta)^\top
\mathcal J(\bm\theta)^{-1}
\mathcal H(\bm\theta),
\end{equation}
where
\[
\mathcal J(\bm\theta) =
\begin{pmatrix}
\mathcal J_{\bm\beta\bm\beta}(\bm\theta) & 0\\
0 & \mathcal J_{\mu\mu}(\bm\theta)
\end{pmatrix},
\]
with
\begin{align}
\mathcal J_{\bm\beta\bm\beta}(\bm\theta) &=
\sum_{k=1}^{N-1}
\widetilde\pi_{\bm\theta}(k)\,\widetilde\lambda_k(\bm\theta)\,
\bar g_{\bm\beta}(k;\bm\theta)\bar g_{\bm\beta}(k;\bm\theta)^\top,
\label{eq:J-beta-corrected}\\
\mathcal J_{\mu\mu}(\bm\theta) &=
\sum_{k=2}^{N}
\widetilde\pi_{\bm\theta}(k)\,\widetilde\mu_k(\bm\theta)\,
\bar g_{\mu}(k;\bm\theta)^2,
\label{eq:J-mu-corrected}
\end{align}
and $\mathcal H(\bm\theta)$ has components
\begin{align}
\mathcal H_{\bm\beta\bm\beta}(\bm\theta) &=
\sum_{k=1}^{N-1}
\widetilde\pi_{\bm\theta}(k)\,\widetilde\lambda_k(\bm\theta)\,
g_{\bm\beta}^{(+)}(k;\bm\theta)\bar g_{\bm\beta}(k;\bm\theta)^\top,
\label{eq:H-beta-beta-corrected}\\
\mathcal H_{\bm\beta\mu}(\bm\theta) &=
\sum_{k=2}^{N}
\widetilde\pi_{\bm\theta}(k)\,\widetilde\mu_k(\bm\theta)\,
g_{\bm\beta}^{(-)}(k;\bm\theta)\bar g_{\mu}(k;\bm\theta),
\label{eq:H-beta-mu-corrected}\\
\mathcal H_{\mu\bm\beta}(\bm\theta) &=
\sum_{k=1}^{N-1}
\widetilde\pi_{\bm\theta}(k)\,\widetilde\lambda_k(\bm\theta)\,
g_{\mu}^{(+)}(k;\bm\theta)\bar g_{\bm\beta}(k;\bm\theta)^\top,
\label{eq:H-mu-beta-corrected}\\
\mathcal H_{\mu\mu}(\bm\theta) &=
\sum_{k=2}^{N}
\widetilde\pi_{\bm\theta}(k)\,\widetilde\mu_k(\bm\theta)\,
g_{\mu}^{(-)}(k;\bm\theta)\bar g_{\mu}(k;\bm\theta).
\label{eq:H-mu-mu-corrected}
\end{align}
Thus $\mathcal J(\bm\theta)$ is block-diagonal because the working score separates birth and death contributions, whereas $\mathcal H(\bm\theta)$ generally is not.

\section{Asymptotic properties}\label{sec:asymptotic}
Throughout this section, we work under Assumption~\ref{ass:irred_and_stab} from Section~\ref{sec:survival}. In addition, we require the following non-degeneracy conditions.

\begin{assumption}[Identifiability of the survival-conditioned model]
\label{ass:identifiability}
For $\bm \theta,\bm \theta' \in \Theta$, if
\[
\widetilde\lambda_k(\bm \theta)=\widetilde\lambda_k(\bm \theta')
\quad \text{for all } k=1,\dots,N-1,
\qquad
\widetilde\mu_k(\bm \theta)=\widetilde\mu_k(\bm \theta')
\quad \text{for all } k=2,\dots,N,
\]
then $\bm \theta= \bm \theta'$.
\end{assumption}

\begin{assumption}[Identifiability for the conditional MLE]\label{ass:ident-mle}
The Fisher information matrix $\mathcal I(\bm\theta_0)$ is positive definite.
\end{assumption}

\begin{assumption}[Identifiability for the QMLE]\label{ass:ident-qmle}
The Godambe information matrix $\mathcal G(\bm\theta_0)$ is positive definite. In particular, this holds if $\mathcal H(\bm\theta_0)$ is non-singular and $\mathcal J(\bm\theta_0)$ is positive definite.
\end{assumption}

\begin{remark}
\label{rem:identifiability}
Assumption~\ref{ass:identifiability} ensures that the true parameter $\bm \theta_0$ is the unique parameter value generating the transformed birth and death intensities of the $Q$-process. Accordingly, in the consistency proof for the conditional MLE, the limiting normalized log-likelihood has a unique maximizer at $\bm \theta_0$; see Appendix \ref{sec:proofs}. For the QMLE, the same assumption identifies the target parameter, while Assumption~\ref{ass:ident-qmle} provides the local non-degeneracy needed to show that $\bm\theta_0$ is an isolated zero of the limiting estimating function.
\end{remark}

The first main limit theorem concerns the estimator $\widetilde{\bm\theta}_T$ defined by the conditional score equation $U_T(\bm\theta)=0$.

\begin{theorem}[Consistency and normality of the MLE under the $Q$-process]\label{thm:clt-mle}
Suppose that Assumptions~\ref{ass:irred_and_stab}, \ref{ass:identifiability} and \ref{ass:ident-mle} hold. Let $K\subset\Theta$ be a sufficiently small compact neighbourhood of $\bm\theta_0$, and let $\widetilde{\bm\theta}_T$ be a solution of the score equation $U_T(\bm\theta)=0$ taking values in $K$. Then, as $T\to\infty$,
\[
\widetilde{\bm\theta}_T \stackrel{P}{\to}\bm\theta_0, \quad 
\sqrt{T}\bigl(\widetilde{\bm\theta}_T-\bm\theta_0\bigr)
\convd
\mathcal N\!\bigl(0,\mathcal I(\bm\theta_0)^{-1}\bigr)
\quad
\text{under }
\quad
\widetilde{\mathbb P}_{\bm\theta_0}.
\]
\end{theorem}

The second limit theorem concerns the quasi-maximum likelihood estimator $\bar{\bm\theta}_T$ defined by the working score equation $\bar U_T(\bm\theta)=0$.

\begin{theorem}[Consistency and normality of the QMLE under the $Q$-process]\label{thm:clt-qmle}
Suppose that Assumptions~\ref{ass:irred_and_stab}, \ref{ass:identifiability} and \ref{ass:ident-qmle} hold. Let $K\subset\Theta$ be a sufficiently small compact neighbourhood of $\bm\theta_0$, and let $\bar{\bm\theta}_T$ be a solution of the working score equation $\bar U_T(\bm\theta)=0$ taking values in $K$. Then, as $T\to\infty$,
\[
\bar{\bm\theta}_T \stackrel{P}{\to} \bm\theta_0, \qquad
\sqrt{T}\bigl(\bar{\bm\theta}_T-\bm\theta_0\bigr)
\convd
\mathcal N\!\bigl(0,\mathcal G(\bm\theta_0)^{-1}\bigr)
\quad
\text{under }\widetilde{\mathbb P}_{\bm\theta_0}.
\]
\end{theorem}

Thus the QMLE is asymptotically normal under the $Q$-process law, with covariance determined by the inverse Godambe information rather than the inverse Fisher information.

Appendix~\ref{app:QvsConditional} shows that, on every fixed time window, the survival-conditioned law is asymptotically approximated by the $Q$-process law. In this sense, the $Q$-process provides the natural time-homogeneous surrogate for survival-conditioned inference.

The proofs of Theorems~\ref{thm:clt-mle} and~\ref{thm:clt-qmle} are given in Appendix~\ref{sec:proofs}. They are based on the martingale representations of the full score and the working score under the $Q$-process. Ergodicity yields the corresponding law of large numbers for the predictable quadratic variation and sensitivity matrices, while a martingale central limit theorem gives the asymptotic normal limits.

\begin{example}
Figure~\ref{fig:comparison_estimator_means} compares the three estimators (unconditional MLE, conditional MLE and conditional QMLE) through their sample means over increasing observation horizons. The sample mean of the naive estimator $\hat{\bm\theta}_T$ remains biased away from the true parameter values (dashed lines), reflecting survival bias. In contrast, the sample means of $\widetilde{\bm\theta}_T$ and $\bar{\bm\theta}_T$ converge to the true values as $T$ grows, in agreement with the consistency predicted by Theorems~\ref{thm:clt-mle} and~\ref{thm:clt-qmle}. \qed
\end{example}

\begin{figure}[tbp]
\centering
\begin{subfigure}{0.495\textwidth}
  \centering
  \caption{Estimating $b_1=N\beta_1$}
  \includegraphics[width=\linewidth]{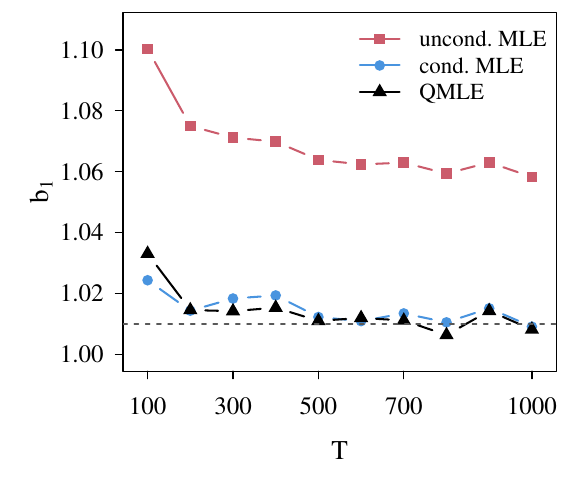}
\end{subfigure}
\hfill
\begin{subfigure}{0.495\textwidth}
  \centering
  \caption{Estimating $b_2=N^2\beta_2$}
  \includegraphics[width=\linewidth]{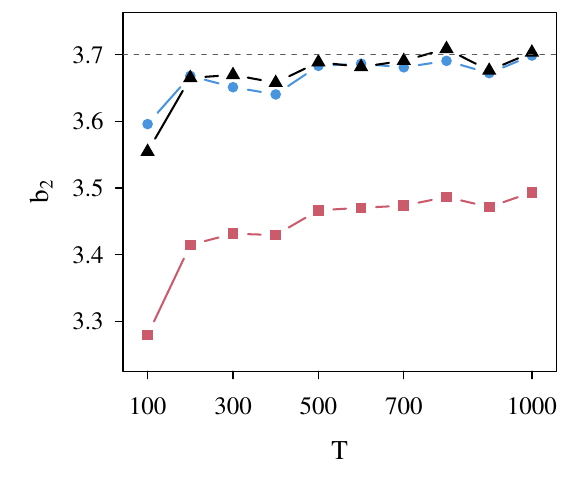}
\end{subfigure}
\caption{Sample means of the three estimators based on $200$ surviving trajectories for each $T\in\{100,200,\ldots,1000\}$, with $N=100$ and $X_0=10$. The plotted coordinates are the scaled parameters $b_1=N\beta_1$ and $b_2=N^2\beta_2$; the true values are $b_1=1.01$ and $b_2=3.70$. The dashed horizontal line marks the true value in each panel. The naive unconditional MLE remains biased, whereas the conditional MLE and the QMLE move toward the truth as the observation horizon increases.}
\label{fig:comparison_estimator_means}
\end{figure}

\subsection{Testing for the presence of a birth mechanism}\label{sec:testing}
A natural application of the asymptotic theory is to test whether a specific birth mechanism is present. We restrict attention to one-dimensional tests, which already cover the main epidemiological question of interest: can a given higher-order transmission effect be detected from the aggregate trajectory alone?

\paragraph{A one-sided Wald (Normal) test based on an unconstrained estimator.}
Fix $i\in\{1,\dots,K\}$ such that the null value $\beta_i =0$ admits an open admissible neighbourhood on which the transformed intensities remain well defined and strictly positive, and the killed generator $Q^+(\bm\theta)$ remains irreducible. For testing, we work with an open admissible parameter set $\Theta_{\mathrm{test}}$ on which the transformed intensities are well defined and strictly positive, and in which the coordinate $\beta_i$ is allowed to take either sign. In particular, the null value $\beta_i=0$ is an interior point of $\Theta_{\mathrm{test}}$. We emphasise that this is only a local enlargement used for the Wald test, while the true epidemiological parameter space remains $\Theta\subset(0,\infty)^{K+1}$. 

We consider the hypotheses
\[
H_{0,i}:\beta_i=0
\qquad \text{versus} \qquad
H_{1,i}:\beta_i>0,
\]
with all remaining parameters treated as nuisance parameters. Let
\[
\widetilde{\bm\theta}_T=
(\widetilde\beta_{1,T},\dots,\widetilde\beta_{K,T},\widetilde\mu_T)
\] 
denote the unconstrained conditional MLE under the $Q$-process model, now viewed as an estimator on the enlarged parameter space $\Theta_{\mathrm{test}}$, and let $\widehat{\mathcal I}_T$ be a consistent estimator of the Fisher information matrix $\mathcal I(\bm\theta_0)$. Under $H_{0,i}$, assuming that the same regularity, identifiability, and non-degeneracy conditions used for Theorem~\ref{thm:clt-mle} hold on a neighbourhood of $\bm\theta_0$ inside $\Theta_{\mathrm{test}}$, we have
\[
\sqrt{T}\bigl(\widetilde{\bm\theta}_T-\bm\theta_0\bigr)
\convd
\mathcal N\!\bigl(0,\mathcal I(\bm\theta_0)^{-1}\bigr).
\]
Hence the studentized statistic
\begin{equation}\label{eq:ZiT}
Z_{i,T} :=
\frac{\widetilde\beta_{i,T}}
{\widehat{\mathrm{se}}(\widetilde\beta_{i,T})},
\qquad
\widehat{\mathrm{se}}(\widetilde\beta_{i,T}) :=
\sqrt{\frac{\bigl(\widehat{\mathcal I}_T^{-1}\bigr)_{ii}}{T}},
\end{equation}
satisfies
\begin{equation}\label{eq:ZiT-null}
Z_{i,T}\convd \mathcal N(0,1)
\qquad\text{under }H_{0,i}.
\end{equation}
Since the alternative is one-sided, we reject $H_{0,i}$ for large positive values of $Z_{i,T}$. Thus an asymptotic level-$\alpha$ test rejects when
\[
Z_{i,T}> z_{1-\alpha},
\]
where $z_{1-\alpha}$ is the $(1-\alpha)$-quantile of the standard normal distribution. The corresponding asymptotic one-sided $p$-value is
\[
p_{i,T}=1-\Phi(Z_{i,T}).
\]
Equivalently, one may use the projected Wald statistic
\begin{equation}\label{eq:WiT}
W_{i,T}:=\bigl(\max\{0,Z_{i,T}\}\bigr)^2.
\end{equation}
By \eqref{eq:ZiT-null}, this statistic has asymptotic null distribution
\begin{equation}\label{eq:boundary-mixture}
W_{i,T}\convd \tfrac12\chi_0^2+\tfrac12\chi_1^2.
\end{equation}
This is simply the distribution of the square of the positive part of a standard normal random variable.

\begin{example}
In the simplicial SIS model, the most relevant case is
\[
H_0:\beta_2=0
\qquad \text{versus} \qquad
H_1:\beta_2>0,
\]
which tests for the presence of triadic transmission. This is the form of the test used in our numerical experiments. Figure~\ref{fig:boundary-testing} summarises the behaviour of the test in this case. Under $H_0$, the standardized statistic $Z_{2,T}$ is asymptotically standard normal, and the equivalent projected statistic $W_{2,T}$ has the $\tfrac12\chi_0^2+\tfrac12\chi_1^2$ limit law. \qed
\end{example}

\begin{figure}[tbp]
\centering
\includegraphics[width=0.46\textwidth]{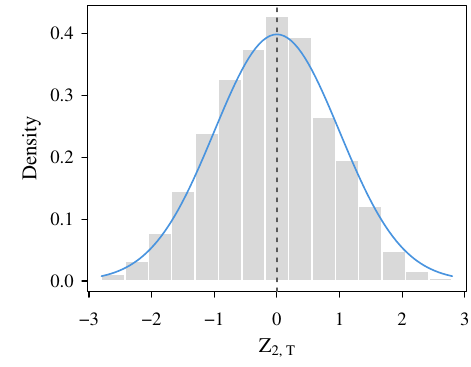}
\caption{Testing $H_0:\beta_2=0$ versus $H_1:\beta_2>0$ in the simplicial SIS model. The figure is based on $1000$ survival-conditioned trajectories of the simplicial SIS model, with $N=100$, $X_0=10$, and observation horizon $T=1000$. The data-generating parameters under $H_0$ are $b_1=N\beta_1=2.875$, $b_2=N^2\beta_2=0$, and $\mu=1$. In each replicate, the standardized statistic $Z_{2,T}$ was computed from the unconstrained conditional MLE together with its Fisher-information-based standard error. The histogram shows the empirical null distribution of $Z_{2,T}$. The blue curve is the standard normal density, and the dashed vertical line marks $0$. The simulated distribution is close to the $\mathcal N(0,1)$ limit predicted by the asymptotic theory.}
\label{fig:boundary-testing}
\end{figure}

\section{Discussion}\label{sec:discussion}
We have developed a likelihood-based framework for inference in composite birth-death processes observed through a single aggregate trajectory. This requires dealing with two main difficulties. First, births can arise through multiple distinct mechanisms, while only the aggregate population size is observed. Second, the presence of an absorbing extinction state means that the available data are typically subject to survival selection. Together, these features make naive inference problematic even in a one-dimensional state space.

Our main result is that the Doob $h$-transform yields a tractable, time-homogeneous surrogate for the original process under long-survival conditioning. This surrogate is the ergodic $Q$-process on the transient states. Working under the $Q$-process law allows us to derive a conditional likelihood, a full score for the conditional MLE $\widetilde{\bm \theta}_T$, and a simpler working score for the QMLE $\bar{\bm \theta}_T$. The resulting asymptotic theory is proved under the $Q$-process law, under which both estimators are consistent and asymptotically normal, while the naive unconditional estimator $\hat{\bm\theta}_T$ remains affected by survival bias. Appendix \ref{app:QvsConditional} explains the relation between the $Q$-process and the original process conditioned on long survival, proving fixed-window convergence and a full-window uniform equivalence statement in the finite-state setting.

A practically useful feature of the present framework is that it suggests a stable latent-variable or EM-type implementation for unmarked births. If birth types are treated as latent variables, then in the conditional expectation step the common Doob-$h$ tilt factor at a given state enters multiplicatively across all birth mechanisms and therefore cancels in the posterior allocation probabilities for the birth labels. This removes the tilt from the relative responsibilities used to split aggregate birth events across mechanisms. However, in an exact EM scheme the maximisation step remains coupled through the compensator terms, since the tilted intensities still depend on the parameter vector via the Doob-$h$ factors. Thus the cancellation is only partial, it simplifies the latent-label update, but does not by itself give a fully decoupled closed-form M-step. Even so, it points to a numerically stable hybrid procedure that may be useful when direct evaluation of the full score is computationally unstable. A detailed algorithmic analysis is left for future work.

The testing problem considered here was deliberately restricted to the one-dimensional case which already addresses the main question of whether a specific higher-order transmission mechanism is present. A natural next step is to treat multi-parameter boundary problems, for example testing the joint presence of several higher-order mechanisms at once. In that setting one expects genuinely non-standard asymptotics, and a possible route is to combine the present survival-conditioned framework with the boundary theory of Self and Liang~\cite{SelfLiang1987}. We do not pursue this here.

Several limitations should also be kept in view. First, we assume continuous observation of the process. Extending the theory to discretely observed data would require handling latent jump times in addition to latent birth types. Second, our main epidemic example is the simplicial SIS model on a complete hypergraph, for which the rates depend only on the current infected count. More general network structures would typically require a higher-dimensional state description, and the one-dimensional reduction exploited here would no longer be available. Third, the asymptotic regime is based on large observation horizons under survival conditioning, and finite-horizon performance must therefore be assessed empirically. 

More broadly, the framework applies beyond epidemic modelling. Any birth-death system with a composite rate structure and an absorbing boundary gives rise to the same basic inferential questions: how to separate competing mechanisms in event space from aggregate observations, and how to account for survival selection when only surviving trajectories are observed. The presented results show that, at least in the finite-state setting considered here, such event-space deconvolution is statistically feasible.

\section*{Acknowledgements}
Marko Lalovic acknowledges financial support from Northeastern University London through a PhD studentship. He also thanks Federico Malizia for discussions during the early stages of this work.

Nicos Georgiou acknowledges partial support from the Dr Perry James (Jim) Browne Research Center at the Department of Mathematics, University of Sussex, and acknowledges financial
support provided by Sapienza University of Rome through the programme Professori Visitatori 2025.

\appendix
\section{Equivalence of survival-conditioned and \texorpdfstring{$Q$}{Q}-process laws}
\label{app:QvsConditional} 
The purpose of this appendix is to justify the use of the Doob $h$-transformed $Q$-process as a time-homogeneous surrogate for the original process conditioned on survival. We compare the
law of the original process conditioned on survival up to a large horizon with the law of the
Doob $h$-transformed Q-process. First we record the standard fixed-window comparison, in
which the sigma-field $\mathcal F_t$ is held fixed while the conditioning horizon $T$ tends to
infinity. We then show that, in the present finite-state setting, the two laws are in fact uniformly equivalent on the full observation window $\mathcal F_T$ for all sufficiently large $T$. For general background on quasi-stationary distributions and $Q$-processes, see~\cite{DarrochSeneta1967,ColletMartinez2013,ChampagnatVillemonais2016}.

Fix an initial state $i \in \{1, \dots, N\}$ and write $\mathbb P_{i,\bm\theta}$ for the law of the original birth-death process started from $X_0 = i$, and $\widetilde{\mathbb P}_{i,\bm\theta}$ for the law of the corresponding $Q$-process started from the same initial state. Both laws are understood on the canonical path space, with natural filtration $\mathcal F_t:=\sigma(X_s:0\le s\le t)$. Let
\be
\mathbb P_{i,\bm\theta}^{(T)} :=
\mathbb P_{i,\bm\theta}(\,\cdot\,\mid \tau_0>T), \qquad \tau_0:=\inf\{ t>0 : X_t = 0\}.
\ee

\begin{proposition}[Fixed-window Radon-Nikodým derivative]
\label{prop:asym-equivalence}
For every fixed $t>0$,
\be\label{eq:RN-derivative}
\left.
\frac{d\mathbb P_{i,\bm\theta}^{(T)}}{d\widetilde{\mathbb P}_{i,\bm\theta}}
\right|_{\mathcal F_t} =
\mathbf 1_{\{\tau_0>t\}}
\frac{h_{\bm\theta}(i)}{h_{\bm\theta}(X_t)}
\frac{\mathbb P_{X_t,\bm\theta}(\tau_0>T-t)}
     {\mathbb P_{i,\bm\theta}(\tau_0>T)}
\,e^{\gamma(\bm\theta)t}.
\ee

Moreover, as $T\to\infty$,
\be\label{eq:RN-to-one}
\left.
\frac{d\mathbb P_{i,\bm\theta}^{(T)}}{d\widetilde{\mathbb P}_{i,\bm\theta}}
\right|_{\mathcal F_t} \longrightarrow 1
\qquad
\widetilde{\mathbb P}_{i,\bm\theta}\text{-a.s.}
\ee
\end{proposition}

\begin{proof}
The identity \eqref{eq:RN-derivative} follows from the Markov property and the definition of the Doob $h$-transform. Indeed, for $A\in\mathcal F_t$,
\be
\mathbb P_{i,\bm\theta}^{(T)}(A) =
\frac{\mathbb E_{i,\bm\theta}\!\left[
\mathbf 1_A\mathbf 1_{\{\tau_0>t\}}
\mathbb P_{X_t,\bm\theta}(\tau_0>T-t)
\right]}
{\mathbb P_{i,\bm\theta}(\tau_0>T)}.
\ee

On the other hand, the Doob $h$-transform gives the standard change-of-measure identity
\be
\left.
\frac{d\widetilde{\mathbb P}_{i,\bm\theta}}{d\mathbb P_{i,\bm\theta}}
\right|_{\mathcal F_t} =
\mathbf 1_{\{\tau_0>t\}}
\frac{h_{\bm\theta}(X_t)}{h_{\bm\theta}(i)}
e^{-\gamma(\bm\theta)t}.
\ee
Combining these two identities yields \eqref{eq:RN-derivative}.

For the limit, standard Perron--Frobenius theory for the killed generator gives the survival asymptotic
\be \label{A.6}
\mathbb P_{x,\bm\theta}(\tau_0>s) =
c_{\bm\theta}\,h_{\bm\theta}(x)e^{\gamma(\bm\theta)s}(1+o(1)),
\qquad s\to\infty,
\ee
uniformly in $x\in\{1,\dots,N\}$; see \cite{DarrochSeneta1967,ColletMartinez2013}. Substituting this into \eqref{eq:RN-derivative} yields \eqref{eq:RN-to-one}.
\end{proof}

\begin{corollary}[Fixed-window contiguity]
\label{cor:contiguity}
For every fixed $t>0$, the family $\{\mathbb P_{i,\bm\theta}^{(T)}|_{\mathcal F_t}:T>t\}$ is contiguous with respect to $\widetilde{\mathbb P}_{i,\bm\theta}|_{\mathcal F_t}$. Equivalently, if $A_T\in\mathcal F_t$ and $\widetilde{\mathbb P}_{i,\bm\theta}(A_T)\to0$, then
\be
\mathbb P_{i,\bm\theta}^{(T)}(A_T) \to 0
\qquad\text{as }T\to\infty.
\ee
\end{corollary}

\begin{proof}
By Proposition~\ref{prop:asym-equivalence},
\[
\mathbb P_{i,\bm\theta}^{(T)}(A_T) =
\widetilde{\mathbb E}_{i,\bm\theta}\!\left[
\mathbf 1_{A_T}
\left.
\frac{d\mathbb P_{i,\bm\theta}^{(T)}}{d\widetilde{\mathbb P}_{i,\bm\theta}}
\right|_{\mathcal F_t}
\right].
\]
Moreover, the Radon--Nikodým derivative is uniformly bounded for all sufficiently large $T$ because the state space is finite and the survival asymptotic used above is uniform in the state variable. That is, there exists a constant $C_t<\infty$ such that, for all sufficiently large $T$,
\[
\left.
\frac{d\mathbb P_{i,\bm\theta}^{(T)}}{d\widetilde{\mathbb P}_{i,\bm\theta}}
\right|_{\mathcal F_t}
\le C_t,
\qquad
\widetilde{\mathbb P}_{i,\bm\theta}\text{-a.s.}
\]
Hence, for all sufficiently large $T$,
\[
\mathbb P_{i,\bm\theta}^{(T)}(A_T)
\le C_t\,\widetilde{\mathbb P}_{i,\bm\theta}(A_T).
\]
Therefore $\mathbb P_{i,\bm\theta}^{(T)}(A_T)\to0$ whenever
$\widetilde{\mathbb P}_{i,\bm\theta}(A_T)\to0$.
\end{proof}

\medskip

The fixed-window statement above explains why the $Q$-process captures the local dynamics of
the original process conditioned on long survival. In the present finite-state setting one can
also strengthen this relation on the full observed window $\mathcal F_T$. The next proposition
records the corresponding Radon--Nikodým derivative and shows that the survival-conditioned
law and the $Q$-process law are uniformly equivalent on $\mathcal F_T$ for all sufficiently large $T$. We emphasise, however, that the asymptotic results in Section~\ref{sec:asymptotic} are still proved under the $Q$-process law $\widetilde{\mathbb P}_{i,\bm\theta}$.

\begin{proposition}[Full-window Radon--Nikodým derivative and uniform equivalence]
\label{prop:A3}
For every $T>0$,
\be
\frac{d\P^{(T)}_{i,\bm\theta}}{d\widetilde \P_{i,\bm \theta}}\Big|_{\mathcal F_T}
=
1_{\{\tau_0>T\}}
\frac{h_{\bm\theta}(i)}{h_{\bm\theta}(X_T)}
\frac{e^{\gamma(\bm\theta)T}}{\P_{i,\bm\theta}(\tau_0>T)}.
\label{A.9}
\ee
Moreover, there exist constants $0<c_1\le c_2<\infty$ and $T_0<\infty$, depending on
$i$ and $\bm\theta$, such that for all $T\ge T_0$,
\be
c_1
\le
\frac{d\P^{(T)}_{i,\bm\theta}}{d\widetilde \P_{i,\bm \theta}}\Big|_{\mathcal F_T}
\le
c_2,
\qquad
\widetilde{\mathbb P}_{i,\bm\theta}\text{-a.s.}
\label{A.10}
\ee
In particular, for every sequence of events $A_T\in\mathcal F_T$,
\be
\widetilde \P_{i,\bm \theta}(A_T)\to 0
\quad\Longleftrightarrow\quad
\P^{(T)}_{i,\bm \theta}(A_T)\to 0.
\label{A.11}
\ee
\end{proposition}

\begin{proof}
The identity \eqref{A.9} is obtained from Proposition \ref{prop:asym-equivalence} by setting $t=T$. Since
$\P_{X_T,\bm \theta}(\tau_0>0)=1$ on $\{\tau_0>T\}$, formula \eqref{eq:RN-derivative} gives
\[
\frac{d\P^{(T)}_{i,\bm \theta}}{d\widetilde \P_{i,\bm \theta}}\Big|_{\mathcal F_T}
=
1_{\{\tau_0>T\}}
\frac{h_{\bm\theta}(i)}{h_{\bm\theta}(X_T)}
\frac{e^{\gamma(\bm \theta)T}}{\P_{i,\bm \theta}(\tau_0>T)}.
\]
To prove \eqref{A.10}, use the survival asymptotic \eqref{A.6} with $x=i$:
\[
\P_{i,\bm \theta}(\tau_0>T)
=
c_{\bm\theta} h_{\bm\theta}(i)e^{\gamma(\bm \theta)T}(1+o(1)),
\qquad T\to\infty.
\]
Hence
\[
\frac{e^{\gamma(\bm \theta)T}}{\P_{i,\bm \theta}(\tau_0>T)}
=
\frac{1+o(1)}{c_{\bm\theta}h_{\bm\theta}(i)}.
\]
Substituting into \eqref{A.9} yields
\[
\frac{d\P^{(T)}_{i,\bm \theta}}{d\widetilde \P_{i,\bm \theta}}\Big|_{\mathcal F_T}
=
1_{\{\tau_0>T\}}
\frac{1+o(1)}{c_{\bm\theta} h_{\bm\theta}(X_T)}.
\]
Since the state space $\{1,\dots,N\}$ is finite and $h_{\bm\theta}(x)>0$ for every
$x\in\{1,\dots,N\}$, the quantity $1/h_{\bm\theta}(X_T)$ is bounded above and below by positive constants uniformly in $T$. Therefore the Radon--Nikodým derivative is bounded above and below by positive constants for all sufficiently large $T$, which proves \eqref{A.10}.

Finally, \eqref{A.11} follows immediately from \eqref{A.10}:
\[
\P^{(T)}_{i,\bm \theta}(A_T)
=
\widetilde \E_{i,\bm \theta}\!\left[
1_{A_T}
\frac{d\P^{(T)}_{i,\bm\theta}}{d\widetilde \P_{i,\bm \theta}}\Big|_{\mathcal F_T}
\right].
\]
The upper and lower bounds in \eqref{A.10} imply
\[
c_1\,\widetilde \P_{i,\bm\theta}(A_T)
\le
\P^{(T)}_{i,\bm \theta}(A_T)
\le
c_2\,\widetilde \P_{i,\bm \theta}(A_T)
\]
for all sufficiently large $T$, and the equivalence follows.
\end{proof}

\begin{sloppypar}
Proposition \ref{prop:A3} shows that, in the finite-state setting considered here, the survival-conditioned law and the $Q$-process law are not only asymptotically equivalent on each fixed window $\mathcal F_t$, but are in fact uniformly equivalent on the full observation window $\mathcal F_T$ for all sufficiently large $T$. This strengthens the interpretation of the $Q$-process as a mathematically tractable surrogate for long-surviving trajectories. The limit theorems in Section \ref{sec:asymptotic} are nevertheless stated and proved under the $Q$-process law.
\end{sloppypar}

\section{Proofs of the asymptotic theorems}\label{sec:proofs}
We work throughout under the $Q$-process law $\widetilde{\mathbb P}_{\bm\theta_0}$. Write $\widetilde X=(\widetilde X_t)_{t \ge 0}$ for the corresponding ergodic birth-death chain on $\{1,\dots,N\}$. The notation for the full score $U_T(\bm\theta)$, the working score $\bar U_T(\bm\theta)$, and the information matrices $\mathcal I(\bm\theta)$, $\mathcal J(\bm\theta)$, and $\mathcal H(\bm\theta)$ is that of Section~\ref{sec:conditional-lik}.

Under the $Q$-process the transformed death rate from state $1$ vanishes:
\[
\widetilde\mu_1(\bm\theta) =
\mu\,r(1)\,\frac{h_{\bm\theta}(0)}{h_{\bm\theta}(1)} = 0.
\]
Hence there are no death jumps from state 1, and all terms involving $\log \widetilde\mu_k(\bm\theta)$ and its derivatives are relevant only for $k = 2,\dots,N$. In what follows, all death-side smoothness bounds, score contributions, and statewise death counts are therefore understood on the set of states $\{2,\dots,N\}$.

Fix a compact neighbourhood $U\subset\Theta$ of $\bm\theta_0$. Since the entries of $Q^+(\bm\theta)$ depend affinely on $\bm\theta$ and, by Assumption~\ref{ass:irred_and_stab}, the principal eigenvalue is simple, standard analytic perturbation theory implies that, after possibly shrinking $U$, the maps $\bm\theta\mapsto \gamma(\bm\theta)$ and $\bm\theta\mapsto \bm h_{\bm\theta}$, $\bm\theta\mapsto \bm v_{\bm\theta}$ may be chosen $C^\infty$ on $U$ under the normalization $\bm v_{\bm\theta}^{\top}\bm h_{\bm\theta}=1$; see Kato~\cite{Kato1995}. In particular, $\bm\theta\mapsto h_{\bm\theta}(k)$ is $C^\infty$ for each $k\in\{1,\dots,N\}$. Shrinking $U$ further if necessary, we may assume that $\widetilde\lambda_k(\bm\theta)$ is bounded away from zero for $k=1,\dots,N-1$ and that $\widetilde\mu_k(\bm\theta)$ is bounded away from zero for $k=2,\dots,N$, uniformly in $\bm\theta\in U$.

Since the state space is finite, all partial derivatives up to order three of $\widetilde\lambda_k(\bm\theta)$ and $\log \widetilde\lambda_k(\bm\theta)$ for $k=1,\dots,N-1$, and of $\widetilde\mu_k(\bm\theta)$ and $\log \widetilde\mu_k(\bm\theta)$ for $k=2,\dots,N$, are uniformly bounded in $(k,\bm\theta)$ on the corresponding index sets. No derivative of $\log \widetilde\mu_1(\bm\theta)$ is needed, since $\widetilde\mu_1(\bm\theta)=0$ identically and the $Q$-process has no death transition from state $1$.

For $k\in\{1,\dots,N\}$, define the statewise waiting times
\begin{equation}\label{eq:statewise-times-proof}
T_k(T) := \int_0^T \mathbf 1_{\{\widetilde X_t=k\}}\,dt.
\end{equation}
For birth jumps, define for $k=1,\dots,N-1$
\begin{equation}\label{eq:statewise-births-proof}
N_k^{(+)}(T) := \int_0^T \mathbf 1_{\{\widetilde X_{t-}=k\}}\,dN_t^{(+)}.
\end{equation}
For death jumps, define for $k=2,\dots,N$
\begin{equation}\label{eq:statewise-deaths-proof}
N_k^{(-)}(T) := \int_0^T \mathbf 1_{\{\widetilde X_{t-}=k\}}\,dN_t^{(-)}.
\end{equation}

By ergodicity of the $Q$-process and the law of large numbers for counting
processes,
\begin{align}
\frac{T_k(T)}{T} &\convas \widetilde\pi_{\bm\theta_0}(k),
\qquad k=1,\dots,N,
\label{eq:statewise-time-lln-proof}\\
\frac{N_k^{(+)}(T)}{T} &\convas
\widetilde\pi_{\bm\theta_0}(k)\widetilde\lambda_k(\bm\theta_0),
\qquad k=1,\dots,N-1,
\label{eq:statewise-birth-lln-proof}\\
\frac{N_k^{(-)}(T)}{T} &\convas
\widetilde\pi_{\bm\theta_0}(k)\widetilde\mu_k(\bm\theta_0),
\qquad k=2,\dots,N.
\label{eq:statewise-death-lln-proof}
\end{align}

These convergences provide the finite-state analogue of the normalized law-of-large-numbers and information conditions used in the large-sample likelihood and estimating-equation theory for counting processes. In the present finite-state setting, jumps are bounded, so the corresponding Lindeberg condition is automatic. Assumption~\ref{ass:ident-mle} gives the positive-definite limit matrix needed for the MLE, while Assumption~\ref{ass:ident-qmle} provides the corresponding non-degeneracy for the QMLE.

\subsection{Consistency of the conditional MLE and the QMLE} 
Before proving the asymptotic normality statements, we record the corresponding consistency arguments in the present finite-state setting. The general counting-process background is provided by \cite[Chapter VI]{ABGK1993}, but here the proofs simplify substantially because the state space is finite and the transformed process is ergodic.

The compact set $U$ fixed above is used only to guarantee uniform smoothness and boundedness of the transformed intensities and score weights. In the consistency arguments below, we work on compact subsets $K\subset U$ on which the maximizer or root is constrained to lie.

\begin{lemma}[Uniform law of large numbers on compact parameter sets]\label{lem:uniform-lln-B}
Let $K \subset U$ be compact. Then, under $\widetilde \P_{\bm \theta_0}$,
\[
\sup_{\bm \theta\in K}\left|
\frac{1}{T}\widetilde\ell_T(\bm\theta)-\Lambda(\bm \theta)
\right|\xrightarrow[T\to\infty]{a.s.}0,
\]
where
\begin{align}
\Lambda(\bm \theta)
&:=
\sum_{k=1}^{N-1}\widetilde\pi_{\bm \theta_0}(k)\widetilde\lambda_k(\bm \theta_0)
\log \widetilde\lambda_k(\bm \theta)
+\sum_{k=2}^{N}\widetilde\pi_{\bm \theta_0}(k)\widetilde\mu_k(\bm \theta_0)
\log \widetilde\mu_k(\bm \theta) \notag\\
&\qquad
-\sum_{k=1}^{N}\widetilde\pi_{\bm \theta_0}(k)
\bigl(\widetilde\lambda_k(\bm \theta)+\widetilde\mu_k(\bm \theta)\bigr).
\label{eq:Lambda-limit}
\end{align}
Moreover,
\[
\sup_{\bm \theta\in K}\left\|
\frac{1}{T}\bar U_T(\bm \theta)-\Psi(\bm \theta)
\right\|\xrightarrow[T\to\infty]{a.s.}0,
\]
where $\Psi(\bm \theta)=\bigl(\Psi^{(\bm\beta)}(\bm \theta),\Psi^{(\mu)}(\bm \theta)\bigr)$ is given by
\begin{align}
\Psi^{(\bm \beta)}(\bm \theta)
&=
\sum_{k=1}^{N-1}\widetilde\pi_{\bm \theta_0}(k)\,
\bar g_{\bm\beta}(k;\bm \theta)\,
\bigl(\widetilde\lambda_k(\bm \theta_0)-\widetilde\lambda_k(\bm \theta)\bigr),\label{eq:Psi-beta}\\
\Psi^{(\mu)}(\bm \theta)
&=
\sum_{k=2}^{N}\widetilde\pi_{\bm \theta_0}(k)\,
\bar g_\mu(k;\bm \theta)\,
\bigl(\widetilde\mu_k(\bm \theta_0)-\widetilde\mu_k(\bm \theta)\bigr).\label{eq:Psi-mu}
\end{align}
\end{lemma}

\begin{proof}
By the statewise decomposition of the likelihood,
\begin{align*}
\widetilde\ell_T(\bm\theta)
&=
\sum_{k=1}^{N-1}N_k^{(+)}(T)\log\widetilde\lambda_k(\bm\theta)
+\sum_{k=2}^{N}N_k^{(-)}(T)\log\widetilde\mu_k(\bm\theta)
-\sum_{k=1}^{N}T_k(T)\bigl(\widetilde\lambda_k(\bm\theta)+\widetilde\mu_k(\bm\theta)\bigr).
\end{align*}
Hence
\begin{align*}
\frac1T\widetilde\ell_T(\bm\theta)-\Lambda(\bm\theta)
&=
\sum_{k=1}^{N-1}
\left(\frac{N_k^{(+)}(T)}{T}-\widetilde\pi_{\bm \theta_0}(k)\widetilde\lambda_k(\bm \theta_0)\right)
\log\widetilde\lambda_k(\bm \theta)\\
&\phantom{xxxxxxx}+
\sum_{k=2}^{N}
\left(\frac{N_k^{(-)}(T)}{T}-\widetilde\pi_{\bm \theta_0}(k)\widetilde\mu_k(\bm \theta_0)\right)
\log\widetilde\mu_k(\bm \theta)\\
&\phantom{xxxxxxxxxxxxx}-
\sum_{k=1}^{N}
\left(\frac{T_k(T)}{T}-\widetilde\pi_{\bm \theta_0}(k)\right)
\bigl(\widetilde\lambda_k(\bm \theta)+\widetilde\mu_k(\bm \theta)\bigr).
\end{align*}
On the compact set $K$, the functions
$\bm\theta\mapsto \widetilde\lambda_k(\bm\theta)$ and
$\bm\theta\mapsto \widetilde\mu_k(\bm\theta)$ are bounded by continuity, while
$\bm\theta\mapsto \log \widetilde\lambda_k(\bm\theta)$ for $k=1,\dots,N-1$ and
$\bm\theta\mapsto \log \widetilde\mu_k(\bm\theta)$ for $k=2,\dots,N$ are bounded, because the corresponding intensities are bounded away from $0$ on $K$. Therefore, taking the supremum over $\bm \theta\in K$ and using
\eqref{eq:statewise-time-lln-proof}--\eqref{eq:statewise-death-lln-proof}, we obtain
\[
\sup_{\bm\theta\in K}\left|
\frac1T\widetilde\ell_T(\bm\theta)-\Lambda(\bm\theta)
\right|\xrightarrow[T\to\infty]{a.s.}0.
\]

Similarly, from \eqref{eq:working-score-vec},
\begin{align*}
\frac1T\bar U_T^{(\bm\beta)}(\bm\theta)
&= \sum_{k=1}^{N-1}\bar g_{\bm\beta}(k;\bm\theta)\frac{N_k^{(+)}(T)}{T}
-\sum_{k=1}^{N-1}\bar g_{\bm\beta}(k;\bm\theta)\widetilde\lambda_k(\bm\theta)\frac{T_k(T)}{T},\\
\frac1T\bar U_T^{(\mu)}(\bm\theta)
&=
\sum_{k=2}^{N}\bar g_\mu(k;\bm\theta)\frac{N_k^{(-)}(T)}{T}
-\sum_{k=2}^{N}\bar g_\mu(k;\bm \theta)\widetilde\mu_k(\bm \theta)\frac{T_k(T)}{T}.
\end{align*}
Since $\bar g_{\bm\beta}(\cdot;\bm\theta)$ and $\bar g_\mu(\cdot;\bm\theta)$ are continuous in
$\theta$ and bounded on $K$, another application of \eqref{eq:statewise-time-lln-proof}--\eqref{eq:statewise-death-lln-proof} yields the
uniform convergence of $T^{-1}\bar U_T(\bm\theta)$ to $\Psi(\bm\theta)$.
\end{proof}

\begin{proposition}[Consistency of the conditional MLE]\label{prop:consistency-mle-B}
Suppose that Assumptions~\ref{ass:irred_and_stab} and
\ref{ass:identifiability} hold. Let $K \subset U$ be a compact neighbourhood of $\bm\theta_0$, and let $\widetilde{\bm \theta}_T$ be any maximiser of $\widetilde\ell_T$ over $K$. Then we have the convergence in probability
\[
\widetilde{\bm \theta}_T \xrightarrow[T\to\infty]{P} \bm\theta_0,
\quad\text{under }\widetilde \P_{\bm\theta_0}.
\]
\end{proposition}

\begin{proof}
By Lemma~\ref{lem:uniform-lln-B},
\[
\sup_{\bm\theta\in K}\left|
\frac1T\widetilde\ell_T(\bm \theta)-\Lambda(\bm \theta)
\right|\xrightarrow[T\to\infty]{a.s.}0.
\]

We now argue that $\Lambda$ is uniquely maximised at $\bm\theta_0$ on $K$. Using \eqref{eq:Lambda-limit}, we compute
\begin{align*}
\Lambda(\bm\theta)-\Lambda(\bm \theta_0)
&=
\sum_{k=1}^{N-1}\widetilde\pi_{\bm\theta_0}(k)
\left[
\widetilde\lambda_k(\bm \theta_0)
\log\frac{\widetilde\lambda_k(\bm \theta)}{\widetilde\lambda_k(\bm \theta_0)}
-\bigl(\widetilde\lambda_k(\bm\theta)-\widetilde\lambda_k(\bm \theta_0)\bigr)
\right]\\
&\quad+
\sum_{k=2}^{N}\widetilde\pi_{\bm\theta_0}(k)
\left[
\widetilde\mu_k(\bm\theta_0)
\log\frac{\widetilde\mu_k(\bm\theta)}{\widetilde\mu_k(\bm\theta_0)}
-\bigl(\widetilde\mu_k(\bm\theta)-\widetilde\mu_k(\bm\theta_0)\bigr)
\right].
\end{align*}
For any $x,y>0$, we have the inequality 
$x\log\frac{y}{x}-(y-x)\le 0,$
with equality if and only if $y=x$. Therefore $\Lambda(\bm\theta)\le \Lambda(\bm\theta_0)$ for all $\bm \theta\in K$. By Assumption~\ref{ass:identifiability}, equality holds only at $\bm\theta=\bm\theta_0$.

Let $\varepsilon>0$. Since $\bm\theta_0$ is the unique maximiser of $\Lambda$ on the compact set $K$, there exists $\eta_\varepsilon>0$ such that
\[
\Lambda(\bm\theta_0)-\Lambda(\bm\theta)\ge \eta_\varepsilon
\qquad\text{for all }\bm\theta\in K\text{ with }\|\bm\theta-\bm\theta_0\|\ge\varepsilon.
\]
For $T$ large, on the high probability event
\[
A_{T, \varepsilon} = \left\{\sup_{\bm\theta\in K}\left|
\frac1T\widetilde\ell_T(\bm\theta)-\Lambda(\bm\theta)
\right|<\frac{\eta_\varepsilon}{3}\right\},
\]
we have, for every $\bm\theta\in K$ with $\|\bm\theta-\bm\theta_0\|\ge\varepsilon$,
\[
\frac1T\widetilde\ell_T(\bm\theta)
\le \Lambda(\bm\theta)+\frac{\eta_\varepsilon}{3}
\le \Lambda(\bm\theta_0)-\frac{2\eta_\varepsilon}{3}
< \frac1T\widetilde\ell_T(\bm\theta_0).
\]
Hence no maximiser over $K$ can lie outside the ball $B(\bm \theta_0,\varepsilon)$. Since, as $T$ grows, the event $A_{T, \varepsilon}$
has probability tending to $1$, and since $\varepsilon$ was arbitrary it follows that
\[
\widetilde{\bm\theta}_T \xrightarrow[T\to\infty]{P}\bm\theta_0.
\]

This is the finite-state version of the standard likelihood consistency argument; compare~\cite[Chapter~VI]{ABGK1993}.
\end{proof}

\begin{remark}
Since $\bm\theta_0$ lies in the interior of $\Theta$ and the maximiser $\widetilde{\bm\theta}_T$ is consistent, it is eventually interior with probability tending to one. Hence any such interior maximiser satisfies the score equation $U_T(\bm\theta)=0$. In this sense, the likelihood maximiser and the score-based definition of the conditional MLE agree asymptotically.
\end{remark}

\begin{remark}
The function $\Lambda$ appearing in the consistency proof is not itself a likelihood. Rather, it is the deterministic limit of the normalized conditional log-likelihood under the true law:
\[
\frac{1}{T}\,\widetilde\ell_T(\bm\theta)\to \Lambda(\bm\theta),
\quad \text{under } \widetilde \P_{\bm\theta_0}.
\]
Thus $\Lambda$ should be viewed as the population contrast associated with the random criterion $\widetilde\ell_T$. The appearance of both $\bm\theta_0$ and $\bm\theta$ is therefore natural: the parameter $\bm\theta_0$ governs the law under which the data are generated, while $\bm\theta$ is the candidate value at which the likelihood is evaluated.

The role of $\Lambda$ in the argument is only to identify the limiting target of the estimator. The estimator itself remains the data-dependent maximizer $\widetilde{\bm\theta}_T$ (or, equivalently here, a solution of the score equation). Consistency follows because the random criterion $T^{-1}\widetilde\ell_T(\bm\theta)$ converges to a deterministic function $\Lambda(\bm\theta)$ having a unique maximizer at $\bm \theta_0$. In particular, $\Lambda$ is a proof device for locating the asymptotic target, not a replacement for the estimator.
\end{remark}

\begin{proposition}[Consistency of the QMLE]\label{prop:consistency-qmle-B}
Suppose that Assumptions~\ref{ass:irred_and_stab}, \ref{ass:identifiability} and \ref{ass:ident-qmle} hold. Let $K \subset U$ be a sufficiently small compact neighbourhood of $\bm\theta_0$. Assume that $\bar{\bm\theta}_T$ is a solution of the working score equation $\bar U_T(\bm \theta)=0$ taking values in $K$. Then
\[
\bar{\bm \theta}_T \xrightarrow[T\to\infty]{P} \bm\theta_0,
\quad\text{under }\widetilde \P_{\bm \theta_0}.
\]
\end{proposition}

\begin{proof}
By Lemma~\ref{lem:uniform-lln-B},
\[
\sup_{\bm\theta\in K}\left\|
\frac1T\bar U_T(\bm\theta)-\Psi(\bm\theta)
\right\|\xrightarrow[T\to\infty]{a.s.}0.
\]
Also, by \eqref{eq:Psi-beta}--\eqref{eq:Psi-mu}, we have $\Psi(\bm\theta_0)=0$.

We next identify the derivative of $\Psi$ at $\bm\theta_0$. Differentiating
\eqref{eq:Psi-beta}--\eqref{eq:Psi-mu} at $\bm\theta_0$, we obtain
\[
D\Psi(\bm\theta_0)=-\mathcal H(\bm \theta_0)^\top.
\]
Since
\[
\mathcal G(\bm \theta_0)=\mathcal H(\bm \theta_0)^\top \mathcal J(\bm \theta_0)^{-1}\mathcal H(\bm \theta_0)
\]
is well defined and positive definite by Assumption~\ref{ass:ident-qmle}, the matrix $\mathcal H(\bm\theta_0)$ must have full rank, and hence is non-singular. Therefore $D\Psi(\bm\theta_0) = -\mathcal H(\bm\theta_0)^\top$ is non-singular. By the inverse function theorem, $\bm\theta_0$ is an isolated zero of $\Psi$. Shrinking $K$ if necessary, we may assume that $\bm\theta_0$ is the unique zero of $\Psi$ on $K$.

Let $\varepsilon>0$. Since $\Psi$ is continuous and has no zero on
$\{\bm \theta\in K:\|\bm \theta- \bm\theta_0\|\ge\varepsilon\}$, compactness gives
\[
\inf\{\|\Psi(\bm \theta)\|:\bm \theta\in K,\ \|\bm \theta-\bm\theta_0\|\ge\varepsilon\}
=: c_\varepsilon >0.
\]
On the event
\[
\sup_{\bm\theta\in K}\left\|
\frac1T\bar U_T(\bm \theta)-\Psi(\bm \theta)
\right\|<\frac{c_\varepsilon}{2},
\]
every root of $\bar U_T(\bm \theta)=0$ lying in $K$ must satisfy
$\|\bm \theta-\bm \theta_0\|<\varepsilon$, since otherwise
\[
0=\left\|\frac1T\bar U_T(\bm \theta)\right\|
\ge \|\Psi(\bm \theta)\|-
\left\|\frac1T\bar U_T(\bm \theta)-\Psi(\bm\theta)\right\|
> \frac{c_\varepsilon}{2},
\]
which is a contradiction. Because the above event has probability tending to $1$ as $T$ grows, and $\varepsilon$ can be taken arbitrarily small, we conclude that
\[
\bar{\bm \theta}_T\xrightarrow[T\to\infty]{P}\bm \theta_0.
\]
This is the finite-state analogue of the standard consistency argument for estimating
functions; compare \cite[Section~VI.2]{ABGK1993}.
\end{proof}

\subsection{Proofs of the asymptotic normality}

\begin{proof}[Proof of Theorem \ref{thm:clt-mle}.]
By Proposition~\ref{prop:consistency-mle-B}, the estimator $\widetilde\theta_T$ is consistent. The conditional log-likelihood $\widetilde\ell_T(\bm\theta)$ is the log-likelihood of a finite-dimensional parametric counting-process model under the $Q$-process. On the birth side, the relevant log-intensities are $\log\widetilde\lambda_k(\bm\theta)$ for $k=1,\dots,N-1$, while on the death side they are $\log\widetilde\mu_k(\bm\theta)$ for $k=2,\dots,N$. The preceding verification shows that the required smoothness, bracket convergence, and remainder bounds hold on these index sets. Therefore, by the standard counting-process likelihood theory~\cite[Chapter VI]{ABGK1993}, the score equation $U_T(\bm\theta)=0$ admits a consistent solution $\widetilde{\bm\theta}_T$, and
\begin{equation}\label{eq:mle-clt-qprocess}
\sqrt{T}\bigl(\widetilde{\bm\theta}_T-\bm\theta_0\bigr)
\convd
\mathcal N\!\bigl(0,\mathcal I(\bm\theta_0)^{-1}\bigr)
\qquad
\text{under }\widetilde{\mathbb P}_{\bm\theta_0}.
\end{equation}
This proves Theorem \ref{thm:clt-mle}.
\end{proof}

\begin{proof}[Proof of Theorem \ref{thm:clt-qmle}]
By Proposition~\ref{prop:consistency-qmle-B}, the estimator $\bar\theta_T$ is consistent.
The working score $\bar U_T(\bm\theta)$ is an estimating function of the usual counting-process form. Its birth component uses predictable weights $\bar g_{\bm\beta}(k;\bm\theta)$ on states $k=1,\dots,N-1$, and its death component uses $\bar g_\mu(k;\bm\theta)$ on states $k=2,\dots,N$. These weights are continuous in $\bm\theta$, locally bounded, and satisfy the same finite-state regularity bounds on $U$ as above, again with the death-side quantities restricted to states $2,\dots,N$.

The general asymptotic theory for such M-estimators, see \cite[Section VI.2]{ABGK1993}, yields consistency of $\bar{\bm\theta}_T$ together with asymptotic covariance of sandwich form. In our notation, the sandwich covariance is
\be\label{eq:qme-sandwich-cov}
\mathcal H(\bm\theta_0)^{-1}
\mathcal J(\bm\theta_0)
\mathcal H(\bm\theta_0)^{-T} =
\mathcal G(\bm\theta_0)^{-1}.
\ee
It follows that
\begin{equation}\label{eq:qmle-clt-qprocess}
\sqrt{T}\bigl(\bar{\bm\theta}_T-\bm\theta_0\bigr)
\convd
\mathcal N\!\bigl(0,\mathcal G(\bm\theta_0)^{-1}\bigr) 
\qquad
\text{under }\widetilde{\mathbb P}_{\bm\theta_0}.
\end{equation}
This concludes the proof of Theorem \ref{thm:clt-qmle}.
\end{proof}

\bibliographystyle{abbrvnat}
\bibliography{main}
\end{document}